\documentclass[dvips,aos,preprint]{imsart}

\RequirePackage[OT1]{fontenc} \RequirePackage{amsthm,amsmath}
\RequirePackage[numbers]{natbib}
\RequirePackage[colorlinks,citecolor=blue,urlcolor=blue]{hyperref}
\RequirePackage{hypernat}
\usepackage{amsmath}
\usepackage{amssymb}
\usepackage{amsmath}
\usepackage{amsfonts}
\usepackage[american]{babel}
\usepackage{graphicx}
\usepackage{flafter}
\usepackage[section]{placeins}
\arxiv{math.PR/0000000}

\startlocaldefs

\def\P{{\mathbb P}}
\def\E{{\mathbb E }}

\def\Bbb E{\mathbb{E}}
\def\Bbb R{\mathbb{R}}
\newtheorem{proposition}{Proposition}
\newtheorem{lemma}{Lemma}
\newtheorem{theorem}{Theorem}
\newtheorem{corollary}{Corollary}

\makeatletter \@addtoreset{equation}{section}

\makeatother

\font\tencmmib=cmmib10 \skewchar\tencmmib '60
\newfam\cmmibfam
\textfont\cmmibfam=\tencmmib

\font\tenmsb=msbm10 
\def\Bbb#1{\hbox{\tenmsb#1}}

\def\lessim{\ \lower4pt\hbox{$
\buildrel{\displaystyle <}\over\sim$}\ }
\def\gessim{\ \lower4pt\hbox{$\buildrel{\displaystyle >}
\over\sim$}\ }

\def\AA{{\cal A} }

\def\go0{\to 0}

\def\leftitem#1{\item{\hbox to\parindent{\enspace#1\hfill}}}

\def\sg{\sigma}

\def\sg2{\sigma^2}

\def\__{_{\infty}}
\def\E{\mathbb E}

\def\R{\mathbb R}
\newtheorem{assumption}{Assumption}

\numberwithin{equation}{section} \theoremstyle{plain}

\endlocaldefs

\begin{document}

\begin{frontmatter}
\title{Nuclear-norm penalization and optimal
rates for noisy low-rank matrix completion} \runtitle{Nuclear-norm
penalization and optimal rates}

\begin{aug}
\author{\fnms{Vladimir} \snm{Koltchinskii}\thanksref{t1,m1}\ead[label=e1]{vlad@math.gatech.edu}},
\author{\fnms{Karim} \snm{Lounici}\thanksref{m1}\ead[label=e2]{klounici@math.gatech.edu}}
\and
\author{\fnms{Alexandre B.} \snm{Tsybakov}\thanksref{t2,m2}
\ead[label=e3]{alexandre.tsybakov@ensae.fr}
\ead[label=u1,url]{http://www.foo.com}}

\thankstext{t1}{Supported in part by NSF grants
DMS-09-06880 and CCF-0808863}
\thankstext{t2}{Supported in part by ANR ``Parcimonie" and by PASCAL-2 Network of Excellence}
\runauthor{V. Koltchinskii, K. Lounici and A.B. Tsybakov}

\affiliation{Georgia Institute of Technology\thanksmark{m1} and
CREST\thanksmark{m2}}

\address{School of Mathematics\\
Georgia Institute of Technology\\
Atlanta, GA 30332-0160\\
\printead{e1}\\
\phantom{E-mail:\ }\printead*{e2}}

\address{CREST \\
3, Av. Pierre Larousse\\
92240 Malakoff, France \\
\printead{e3}\\}
\end{aug}

\begin{abstract} \
This paper deals with the trace regression model where $n$ entries
or linear combinations of entries of an unknown $m_1\times m_2$
matrix $A_0$ corrupted by noise are observed. We propose a new
nuclear-norm penalized estimator of $A_0$ and establish a general
sharp oracle inequality for this estimator for arbitrary values of
$n,m_1,m_2$ under the condition of isometry in expectation. Then
this method is applied to the matrix completion problem. In this
case, the estimator admits a simple explicit form and we prove that
it satisfies oracle inequalities with faster rates of convergence
than in the previous works. They are valid, in particular, in the
high-dimensional setting $m_1m_2\gg n$. We show that the obtained
rates are optimal up to logarithmic factors in a minimax sense and
also derive, for any fixed matrix $A_0$, a non-minimax lower bound
on the rate of convergence  of our estimator, which coincides with
the upper bound up to a constant factor. Finally, we show that our
procedure provides an exact recovery of the rank of $A_0$ with
probability close to 1. We also discuss the statistical learning
setting where there is no underlying model determined by $A_0$ and
the aim is to find the best trace regression model approximating the
data. As a by-product, we show that, under the Restricted Eigenvalue
condition, the usual vector Lasso estimator satisfies a sharp oracle
inequality (i.e., an oracle inequality with leading constant 1).
\end{abstract}

\begin{keyword}[class=AMS]
\kwd[Primary ]{62J99,62H12} \kwd[; secondary ]{60B20, 60G15}
\end{keyword}

\begin{keyword}
\kwd{matrix completion} \kwd{low-rank matrix estimation}
\kwd{recovery of the rank} \kwd{statistical learning}\kwd{optimal
rate of convergence}\kwd{noncommutative Bernstein inequality}
\kwd{Lasso}
\end{keyword}

\end{frontmatter}

\section{Introduction}
Assume that we observe $n$ independent random pairs $(X_i,Y_i),
i=1,\dots,n$, where $X_i$ are random matrices with dimensions
$m_1\times m_2$ and $Y_i$ are random variables in $\R$, satisfying
the trace regression model:
\begin{equation}\label{tracereg}
\E(Y_i|X_i) = \mathrm{tr}(X_i^\top A_0), \quad i=1,\dots,n,
\end{equation}
where $A_0 \in \mathbb R^{m_1\times m_2}$ is an unknown matrix,
$\E(Y_i|X_i)$ is the conditional expectation of $Y_i$ given $X_i$,
and $\mathrm{tr}(B)$ denotes the trace of matrix $B$. We consider
the problem of estimation of $A_0$ based on the observations
$(X_i,Y_i), i=1,\dots,n$. Though the results of this paper are
obtained for general $n,m_1,m_2$, the main motivation is in the
high-dimensional setting, which corresponds to $m_1m_2\gg n$, with
low-rank matrices $A_0$.

It will be convenient to write the model (\ref{tracereg}) in the
form
\begin{equation}\label{tracereg1}
Y_i = \mathrm{tr}(X_i^\top A_0) + \xi_i, \quad i=1,\dots,n,
\end{equation}
where the noise variables $\xi_i=Y_i-\E(Y_i|X_i)$ are independent
and have zero means.

For any matrices $A,B\in \R^{m_1\times m_2}$, we define the scalar
product
$$\langle A,B \rangle = \mathrm{tr}(A^\top B)
$$
and the bilinear form
$$\langle A,B \rangle_{L_2(\Pi)} = \frac{1}{n}\sum_{i=1}^n
{\mathbb E}\big(\langle A,X_i\rangle \langle B,X_i\rangle\big)\,.
$$
Here $\Pi= \frac{1}{n}\sum_{i=1}^n \Pi_i$, where $\Pi_i$ denotes the
distribution of $X_i$. The corresponding semi-norm $\|A
\|_{L_2(\Pi)}$ is given by $$\|A \|_{L_2(\Pi)}^2 =
\frac{1}{n}\sum_{i=1}^n {\mathbb E}\big(\langle
A,X_i\rangle^2\big)\,.
$$

\textbf{Example 1. Matrix Completion.} Assume that the design
matrices $X_i$ are i.i.d. uniformly distributed on the set
\begin{equation}
\mathcal{X} = \left\{ e_j(m_1) e_k^\top (m_2), 1\leq j \leq m_1,\,
1\leq k \leq m_2 \right\},
\end{equation}
where $e_k(m)$ are the canonical basis vectors in $\R^{m}$. The set
$\mathcal{X}$ forms an orthonormal basis in the space of $m_1\times
m_2$ matrices that will be called the matrix completion basis. Let
also $n< m_1m_2$. Then the problem of estimation of $A_0$ coincides
with the problem of matrix completion under uniform sampling at
random (USR) as studied in the non-noisy case ($\xi_i=0$) in
\cite{Gross-2,Recht}, and in the noisy case in
\cite{Rohde,Gaiffas_Lecue}. Considering low-rank matrices $A_0$ is
of a particular interest. Clearly, for such $X_i$ we have the
isometry
\begin{equation}\label{isom1}
\|A\|_{L_2(\Pi)}^2= {\mu}^{-2}\|A\|_2^2,
\end{equation}
 for all matrices $A\in
\R^{m_1\times m_2}$, where $\mu = \sqrt{m_1 m_2}$, and $\|A\|_{2}$
is the Frobenius norm of $A$. However, the restricted isometry
property in the usual sense, i.e., "in probability", cf., e.g.,
\cite{rfp07}, does not hold for matrix completion, since for $n<
m_1m_2$ there trivially exists a matrix of rank 1 in the null space
of the sampling operator.

One can also consider more general matrix measurement models in
which, for a given orthonormal basis in the space of matrices, a
random sample of Fourier coefficients of the target matrix $A_0$ is
observed subject to a random noise. For more discussion on matrix
completion with other types of sampling, see
\cite{cp09,Candes_Recht,Candes_Tao,Keshavan,Kol10} and references
therein.

\medskip

\textbf{Example 2. Column masks.} Assume that the design matrices
$X_i$ are i.i.d. replications of a random matrix $X$, which has only
one nonzero column. For instance, let the distribution of $X$ be
such that all the columns have equal probability to be non-zero, and
the random entries of non-zero column $x_{(j)}$ are such that
$\E(x_{(j)}x_{(j)}^\top)$ is the identity matrix. Then
$\|A\|_{L_2(\Pi)}^2= \|A\|_2^2/m_2,\, \ \forall A\in \R^{m_1\times
m_2}$, so that condition (\ref{isom1}) is satisfied with $\mu =
\sqrt{m_2}$. More generally, in view of application to multi-task
learning, cf. \cite{Rohde}, one can be interested in considering
non-identically distributed $X_i$. The model can be then
reformulated as a longitudinal regression model, with different
distributions of $X_i$ corresponding to different tasks.

\medskip

\textbf{Example 3. "Complete" subgaussian design.} Assume that the
design matrices $X_i$ are i.i.d. replications of a random matrix $X$
such that $\langle A,X \rangle$ is a subgaussian random variable for
any $A\in \R^{m_1\times m_2}$. This approach has its roots in
compressed sensing. The two major examples are given by the matrices
$X$ whose entries are either i.i.d. standard Gaussian or Rademacher
random variables. In both cases, we have $\|A\|_{L_2(\Pi)}^2=
\|A\|_2^2,\, \ \forall A\in \R^{m_1\times m_2}$, so that condition
(\ref{isom1}) is satisfied with $\mu = 1$. The problem of exact
reconstruction of $A_0$ under such a design in the non-noisy setting
was studied in \cite{rfp07,Candes_Plan,nw09}, whereas estimation of
$A_0$ in the presence of noise is analyzed in
\cite{nw09,Rohde,Candes_Plan}, among which \cite{Rohde,Candes_Plan}
treat the high-dimensional case $m_1m_2>n$.

\medskip

\textbf{Example 4. Fixed design.} Assume that all the $\Pi_i$ are
Dirac measures, so that the design matrices $X_i$ are non-random.
Then $\|A\|_{L_2(\Pi)}^2= \frac{1}{n}\sum_{i=1}^n \langle
A,X_i\rangle^2,$ and we get the problem of trace regression with
fixed design, cf. \cite{Rohde}. In particular, if $m_1=m_2$, and $A$
and $X_i$ are diagonal matrices the trace regression model
(\ref{tracereg1}) becomes the usual linear regression model.
Accordingly, the rank of $A$ becomes the number of its non-zero
diagonal elements. This observation will allow us to deduce, as a
consequence of our general argument, an oracle inequality for the
usual Lasso in sparse linear regression with fixed design improving
\cite{BRT09} in the sense that the inequality is sharp (cf. Theorem
\ref{th:main_RE} and Section \ref{subsec:lasso}).

The general oracle inequalities that we will prove in Section
\ref{oracle} can be successfully applied to the above examples. The
emphasis in this paper will be on the matrix completion problem
(Example 1), for which the previously obtained results were
suboptimal.

Statistical estimation of low-rank matrices has recently become a
very active field with a rapidly growing literature. The most
popular methods are based on penalized empirical risk minimization
with nuclear-norm penalty
\cite{argyr07,argyr08,argyr09,bach08,bsw10,cp09,Candes_Plan,Gaiffas_Lecue,nw09,nw10,Rohde}.
Estimators with other types of penalization, such as the
Schatten-$p$ norm \cite{Rohde}, the von Neumann entropy
\cite{Kol10}, penalization by the rank \cite{bsw10,gir10} or some
combined penalties \cite{Gaiffas_Lecue} are also discussed.

It is worth pointing out that in many applications, such as in
matrix completion, the distribution $\Pi$ is known, and yet this
information has not been exploited since the penalized estimation
procedures considered in the literature involve the empirical risk
$\frac{1}{n}\sum_{i=1}^n (Y_i - \mathrm{tr}(X_i^\top A))^2$. %(cf.
%\cite{Rohde,Gaiffas_Lecue}).
In this paper we incorporate the
knowledge of $\Pi$ in the construction and we study the following
estimator of $A_0$:
\begin{equation}
\label{ERM} {\hat A}^\lambda = \displaystyle{{\rm argmin}_{A\in
\mathbb A }} L_n(A),
\end{equation}
where $\mathbb A\subseteq {\mathbb R}^{m_1\times m_2}$ is a set of
matrices,
\begin{equation}
\label{ER} L_n(A)=\|A\|_{L_2(\Pi)}^2 -\biggl\langle
\frac{2}{n}\sum_{i=1}^n Y_i X_i, A \biggr\rangle + \lambda \|A\|_1,
\end{equation}
$\lambda>0$ is a regularization parameter, and $\|A\|_1$ is the
nuclear norm of $A$. We will mainly consider convex sets $\mathbb
A$. Note that if all $X_i$ are non-random, ${\hat A}^\lambda$
coincides with the usual matrix Lasso estimator:
\begin{equation}
\label{ER_fixed} \hat A^{\lambda}={\rm argmin}_{A\in {\mathbb
A}}\biggl[ n^{-1}\sum_{j=1}^n (Y_j-\langle A,X_j\rangle)^2
 + \lambda \|A\|_1\biggr].
\end{equation}

The emphasis in this paper is on the noisy matrix completion
setting. Then the estimator ${\hat A}^\lambda$ has a particularly
simple form; it is obtained from the matrix $
\frac{m_1m_2}{n}\sum_{i=1}^n Y_i X_i$ by soft thresholding of its
singular values. One of the main results of this paper is to show
that our estimators are rate optimal (up to logarithmic factors)
under the Frobenius error for a simple class of matrices ${\cal
A}(r,a)$ defined by two restrictions: the rank of $A_0$ is not
larger than given $r$ and all the entries of $A_0$ are bounded in
absolute value by a constant $a$. This rather intuitive class has
been first considered in \cite{Keshavan}. However, the construction
of the estimator in \cite{Keshavan} requires the exact knowledge of
${\rm rank}( A_0)$ and the upper bound on the Frobenius error
obtained in \cite{Keshavan} is suboptimal (see the details in
Section \ref{sec:matrix_compl}). The recent paper
\cite{Gaiffas_Lecue} obtains suboptimal bounds of "slow rate" type
for matrix completion while \cite{Kol10} focuses on complex-valued
Hermitian matrices with nuclear norm equal to~1, which is motivated
by density matrix estimation problem in quantum state tomography.
These papers do not address the optimality issue. Optimal rates in
noisy matrix completion are derived in \cite{Rohde}, but on
different classes of matrices and with the empirical prediction
error rather than with the Frobenius error. Finally, \cite{nw10}
discusses the optimality issue for the Frobenius error on the
classes defined in terms of a "spikiness index" of $A_0$, which are
not related to ${\cal A}(r,a)$, and suggests estimators that require
prior knowledge about this index.

The main contributions of this paper are the following. In Section
\ref{oracle} we derive a general oracle inequality for the
prediction error $\|{\hat A}^\lambda - A_0\|_{L_2(\Pi)}^2$. This
oracle inequality is sharp, i.e., with leading constant 1, both in
the case of "slow rate" (for matrices $A_0$ with small nuclear norm)
and in the case of "fast rate" (for matrices $A_0$ with small rank).
As a particular instance of this general result, in Section
\ref{sec:matrix_compl} we obtain an oracle inequality for the matrix
completion problem. In Section \ref{lower bounds}, we establish
minimax lower bounds showing that the rates for matrix completion
obtained in Section \ref{sec:matrix_compl} are optimal up to a
logarithmic factor. In Section \ref{sec:discussion}, we briefly
discuss some other implications and extensions of our method.
Finally, Section \ref{sto} is devoted to the control of the
stochastic term appearing in the proof of the upper bound.

\section{General oracle inequalities}\label{oracle}

We recall first some basic facts about matrices. Let
$A\in\R^{m_1\times m_2}$ be a rectangular matrix, and let $r={\rm
rank} (A) \leq \min(m_1,m_2)$ denote its rank. The singular value
decomposition (SVD) of $A$ has the form: $A=\sum_{j=1}^r \sigma_j(A)
u_j v_j^\top$ with orthonormal vectors $u_1,\dots, u_r\in {\mathbb
R}^{m_1}$, orthonormal vectors $v_1,\ldots,v_r\in {\mathbb R}^{m_2}$
and real numbers $\sigma_1(A) \ge\dots\ge \sigma_r(A)>0$ (the
singular values of $A$). The pair of linear vector spaces
$(S_1,S_2)$ where $S_1$ is the linear span of $\{u_1,\dots, u_r\}$
and $S_2$ is the linear span of $\{v_1,\dots, v_r\}$ will be called
the \it support \rm of $A.$ We will denote by $S_j^{\perp}$ the
orthogonal complements of $S_j$, $j=1,2$, and by $P_{S}$ the
projector on the linear vector subspace $S$ of ${\mathbb R}^{m_j}$,
$j=1,2$.

The Schatten-$p$ (quasi-)norm $\| A \|_p$  of matrix $A$ is defined
by
$$
\| A \|_p = \left( \sum_{j=1}^{\min(m_1,m_2)}\sigma_j(A)^p
\right)^{1/p} \ \text{for} \ 0< p<\infty,\quad \text{and}\quad
\|A\|_{\infty}=\sigma_1(A).
$$
Recall the well-known {\it trace duality} property:
$$
\left| \mathrm{tr}(A^\top B) \right| \leq \|A\|_1 \|B\|_{\infty},
\quad \forall A,B\in \R^{m_1\times m_2}.
$$
We will also use the fact that the subdifferential of the convex
function $A\mapsto \|A\|_1$ is the following set of matrices:
\begin{equation}
\label{subdiff}
\partial \|A\|_1=\Bigl\{\sum_{j=1}^r u_j v_j^\top +
P_{S_1^{\perp}}W P_{S_2^{\perp}}:\ \|W\|_\infty\leq 1 \Bigr\}
\end{equation}
(cf. \cite{watson}). Define the random matrix
\begin{equation}\label{randM}
 {\bf M}=
\frac{1}{n}\sum_{i=1}^n (Y_i X_i-{\mathbb E}(Y_iX_i)).
\end{equation}
We will need the following assumption on the distribution of the
matrices $X_i$.
\begin{assumption}\label{design_assumption}
There exists a constant $\mu>0$ such that, for all matrices $A\in
{\mathbb A}-{\mathbb A}:=\{A_1-A_2: A_1,A_2\in \mathbb A\}$,
$$
\|A\|_{L_2(\Pi)}^2\geq {\mu}^{-2}\|A\|_2^2.
$$
\end{assumption}
As discussed in the Introduction, Assumption \ref{design_assumption}
is satisfied, often with equality and for ${\mathbb A}={\mathbb
A}-{\mathbb A}= \R^{m_1\times m_2}$, in several interesting
examples. The next theorem plays the key role in what follows.
\begin{theorem}\label{th:main}
Let $\mathbb A\subseteq \R^{m_1\times m_2}$ be any set of matrices.
If $\lambda \geq 2\|{\bf M}\|_{\infty},$ then
\begin{equation}
\label{first} \|{\hat A}^\lambda-A_0\|_{L_2(\Pi)}^2 \leq \inf_{A\in
{\mathbb A}}\Bigl[\|A-A_0\|_{L_2(\Pi)}^2 + 2 \lambda
\|A\|_1\Bigr]\,.
\end{equation}
If, in addition, $\mathbb A$ is a convex set and Assumption 1 is
satisfied, then
\begin{equation}
\label{third}
%&&
\|{\hat A}^\lambda-A_0\|_{L_2(\Pi)}^2 \leq
%\nonumber
%\\
%&&
\inf_{A\in {\mathbb A}}\Bigl[\|A-A_0\|_{L_2(\Pi)}^2
+\left(\frac{1+\sqrt{2}}{2}\right)^2\mu^2\lambda^2 {\rm
rank}(A)\Bigr]\,.
\end{equation}
Furthermore, in this case for all $A\in {\mathbb A}$ with support
$(S_1,S_2)$,
\begin{eqnarray}
\label{second}
%&&
&&\|{\hat A}^\lambda-A_0\|_{L_2(\Pi)}^2+ (\lambda-2\|{\bf
M}\|_{\infty})
\|P_{S_1^{\perp}}{\hat A}^\lambda P_{S_2^{\perp}}\|_1 \\
&& \quad \leq \nonumber \|A-A_0\|_{L_2(\Pi)}^2
+\left(\frac{1+\sqrt{2}}{2}\right)^2\mu^2\lambda^2 {\rm rank}(A).
\end{eqnarray}
\end{theorem}

\begin{proof} It follows from the definition of the estimator $\hat
A$ that, for all $A\in {\mathbb A}$,
$$L_n({\hat A}^\lambda)=\|{\hat A}^\lambda\|_{L_2(\Pi)}^2 -\biggl\langle
\frac{2}{n}\sum_{i=1}^n Y_i X_i, {\hat A}^\lambda \biggr\rangle +
\lambda \|{\hat A}^\lambda\|_1 \leq
$$
$$
\|A\|_{L_2(\Pi)}^2 -\biggl\langle \frac{2}{n}\sum_{i=1}^n Y_i X_i, A
\biggr\rangle + \lambda \|A\|_1= L_n(A).
$$
Also, note that
$$
\frac{1}{n}\sum_{i=1}^n {\mathbb
E}(Y_iX_i)=\frac{1}{n}\sum_{i=1}^n{\mathbb E}\big(\langle
A_0,X_i\rangle X_i\big)\ \ {\rm and}\ \
\frac{1}{n}\sum_{i=1}^n\langle {\mathbb E}(Y_iX_i),A\rangle =
\langle A_0, A\rangle_{L_2(\Pi)}.
$$
Therefore, we have
$$
\|{\hat A}^\lambda\|_{L_2(\Pi)}^2 -2 \langle {\hat A}^\lambda,A_0
\rangle_{L_2(\Pi)} \leq \|A\|_{L_2(\Pi)}^2 -2 \langle
A,A_0\rangle_{L_2(\Pi)} +
$$
$$
\biggl\langle \frac{2}{n}\sum_{i=1}^n (Y_i X_i-{\mathbb E}(Y_iX_i)),
{\hat A}^\lambda-A \biggr\rangle + \lambda (\|A\|_1 -\|{\hat
A}^\lambda\|_1),
$$
which implies, due to the trace duality,
$$
\|{\hat A}^\lambda-A_0\|_{L_2(\Pi)}^2 \leq \|A-A_0\|_{L_2(\Pi)}^2
+2\Delta \|{\hat A}^\lambda-A\|_1 + \lambda (\|A\|_1 -\|{\hat
A}^\lambda\|_1),
$$
where we set for brevity $\Delta=\|{\bf M}\|_{\infty}$. Under the
assumption $\lambda\geq 2\Delta$ this yields
$$
\|{\hat A}^\lambda-A_0\|_{L_2(\Pi)}^2 \leq \|A-A_0\|_{L_2(\Pi)}^2 +
\lambda(\|{\hat A}^\lambda-A\|_1+ \|A\|_1 -\|{\hat A}^\lambda\|_1)
\leq \|A-A_0\|_{L_2(\Pi)}^2 + 2 \lambda \|A\|_1,
$$
and the bound (\ref{first}) follows.

To prove the remaining bounds, note that a necessary condition of
extremum in problem (\ref{ERM}) implies that there exists $\hat V\in
\partial \|{\hat A}^\lambda\|_1$ such that, for all $A\in {\mathbb
A}$,
\begin{equation}
\label{extremum} 2\langle {\hat A}^\lambda, \hat
A^{\lambda}-A\rangle_{L_2(\Pi)}
-\biggl\langle\frac{2}{n}\sum_{i=1}^n Y_i X_i, {\hat A}^\lambda-A
\biggr\rangle + \lambda \langle \hat V, \hat A^{\lambda}-A\rangle
\leq 0.
\end{equation}
Indeed, since $\hat A^{\lambda}$ is a minimizer of $L_n(A)$ in
${\mathbb A},$ there exists a matrix $B\in \partial L_n(\hat
A^{\lambda})$ such that $-B$ belongs to the normal cone of ${\mathbb
A}$ at the point $\hat A^{\lambda}$ (cf. \cite{Aubin}, Chapter 4,
Section 2, Corollary 6). It is easy to see that $B$ can be
represented as follows
$$
B=2 \int_{{\mathbb R}^{m_1\times m_2}}\langle{\hat A}^\lambda,
X\rangle X\Pi(dX) -\frac{2}{n}\sum_{i=1}^n Y_i X_i + \lambda \hat V,
$$
where $\hat V\in \partial\|\hat A^{\lambda}\|_1.$ The condition that
$-B$ belongs to the normal cone at the point $\hat A^{\lambda}$
implies that $\langle B, \hat A^{\lambda}-A\rangle \leq 0,$ and
(\ref{extremum}) follows.

Consider an arbitrary $A\in {\mathbb A}$ of rank $r$ with spectral
representation $A=\sum_{j=1}^r \sigma_j u_j v_j^\top$ and with
support $(S_1,S_2).$ It follows from (\ref{extremum}) that
\begin{equation}
\label{extremum_A} 2\langle {\hat A}^\lambda-A_0, \hat
A^\lambda-A\rangle_{L_2(\Pi)} +\lambda \langle \hat V-V,{\hat
A}^\lambda-A\rangle \leq -\lambda \langle V,{\hat
A}^\lambda-A\rangle +2\langle {\bf M}, {\hat A}^\lambda-A \rangle
\end{equation}
for an arbitrary $V\in \partial \|A\|_1.$ By monotonicity of
subdifferentials of convex functions, $\langle \hat V-V,{\hat
A}^\lambda-A\rangle \geq 0.$ On the other hand, by (\ref{subdiff}),
the following representation holds
$$
V=\sum_{j=1}^r u_j v_j^\top + P_{S_1^{\perp}}W P_{S_2^{\perp}},
$$
where $W$ is an arbitrary matrix with $\|W\|_\infty\leq 1.$ It
follows from the trace duality %between the norms $\|\cdot\|_1$ and
%$\|\cdot\|_{\infty}$
that there exists $W$ with $\|W\|_{\infty}\leq 1$ such that
$$
\langle P_{S_1^{\perp}}W P_{S_2^{\perp}}, \hat A^{\lambda}-A\rangle
= \langle P_{S_1^{\perp}}W P_{S_2^{\perp}}, \hat A^{\lambda}\rangle
=\langle W , P_{S_1^{\perp}}\hat A^{\lambda}P_{S_2^{\perp}}\rangle =
\|P_{S_1^{\perp}}\hat A^{\lambda}P_{S_2^{\perp}}\|_1,
$$
where in the first equality we used that $A$ has the support
$(S_1,S_2).$ For this particular choice of $W,$ (\ref{extremum_A})
implies that
\begin{equation}
\label{extremum_A'} 2\langle {\hat A}^\lambda-A_0, \hat
A^\lambda-A\rangle_{L_2(\Pi)} +\lambda \|P_{S_1^{\perp}}\hat
A^{\lambda}P_{S_2^{\perp}}\|_1 \leq -\lambda \biggl\langle
\sum_{j=1}^r u_j v_j^\top,{\hat A}^\lambda-A\biggr\rangle +2\langle
{\bf M}, {\hat A}^\lambda-A \rangle.
\end{equation}
Using the identity
\begin{equation}\label{identity}
2\langle {\hat A}^\lambda-A_0, {\hat A}^\lambda-A\rangle_{L_2(\Pi)}=
\|{\hat A}^\lambda-A_0\|_{L_2(\Pi)}^2+ \|{\hat
A}^\lambda-A\|_{L_2(\Pi)}^2- \|A-A_0\|_{L_2(\Pi)}^2
\end{equation}
and the facts that
\begin{equation}\label{uv}\|\sum_{j=1}^r
u_j v_j^\top\|_\infty=1,\quad \biggl\langle \sum_{j=1}^r u_j
v_j^\top,{\hat A}^\lambda-A\biggr\rangle = \biggl\langle
\sum_{j=1}^r u_j v_j^\top,P_{S_1}({\hat
A}^\lambda-A)P_{S_2}\biggr\rangle
\end{equation}
 we deduce from
(\ref{extremum_A'}) that
\begin{eqnarray}
\label{extremum_B} && \|{\hat A}^\lambda-A_0\|_{L_2(\Pi)}^2+ \|{\hat
A}^\lambda-A\|_{L_2(\Pi)}^2 +\lambda \|P_{S_1^{\perp}} {\hat
A}^\lambda P_{S_2^{\perp}}\|_1 \leq \nonumber
\\
&& \|A-A_0\|_{L_2(\Pi)}^2 +\lambda \|P_{S_1}({\hat
A}^\lambda-A)P_{S_2}\|_1 +2 \langle {\bf M}, {\hat A}^\lambda-A
\rangle.
\end{eqnarray}
To provide an upper bound on $2\langle {\bf M}, {\hat A}^\lambda-A
\rangle $ we use the following decomposition
\begin{eqnarray*} \langle {\bf
M}, {\hat A}^\lambda-A \rangle &=& \langle {\cal P}_A({\bf M}),
{\hat A}^\lambda-A \rangle
+ \langle P_{S_1^{\perp}}{\bf M}P_{S_2^{\perp}},{\hat A}^\lambda-A \rangle\\
&=& \langle {\cal P}_A({\bf M}), {\cal P}_A({\hat A}^\lambda-A)
\rangle + \langle P_{S_1^{\perp}}{\bf M}P_{S_2^{\perp}},{\hat
A}^\lambda\rangle,
\end{eqnarray*}
where ${\cal P}_A({\bf M})= {\bf M} -P_{S_1^{\perp}}{\bf
M}P_{S_2^{\perp}}$. This implies, due to the trace duality,
\begin{eqnarray}\label{M0}
2\big|\langle {\bf M}, {\hat A}^\lambda-A \rangle\big| &\leq&
\Lambda \|{\cal P}_A({\hat A}^\lambda-A)\|_2 + \Gamma
\|P_{S_1^{\perp}} {\hat A}^\lambda P_{S_2^{\perp}}\|_1\\
&\leq& \Lambda \|{\hat A}^\lambda-A\|_2 + \Gamma \|P_{S_1^{\perp}}
{\hat A}^\lambda P_{S_2^{\perp}}\|_1,\nonumber%\label{M}
\end{eqnarray}
where
\begin{equation}\label{2.10a}
\Lambda=2\|{\cal P}_A({\bf M})\|_2, \ \ \
\Gamma=2\|P_{S_1^{\perp}}{\bf M} P_{S_2^{\perp}}\|_\infty.
\end{equation}
Note that
\begin{equation}\label{2.10b}
\Gamma \leq 2 \|{\bf M}\|_\infty=2\Delta.
\end{equation}
Since
\begin{equation}\label{2.10c}
{\cal P}_A({\bf M})=P_{S_1^{\perp}}{\bf M} P_{S_2}+P_{S_1}{\bf M}
\end{equation}
and ${\rm rank}(P_{S_j})\le {\rm rank}(A)$, $j=1,2$, we have
$$
\Lambda \leq 2\sqrt{{\rm rank}({\cal P}_A({\bf M}))} \,\|{\bf
M}\|_\infty \le 2\sqrt{2 \,{\rm rank}(A)}\,\Delta \le \sqrt{2 \,{\rm
rank}(A)}\, \lambda.
$$
Due to the fact that
\begin{equation}\label{vnutr}
\|P_{S_1}(\hat A^{\lambda}-A)P_{S_2}\|_1 \leq
\sqrt{\mathrm{rank}(A)}\|P_{S_1}(\hat A^{\lambda}-A)P_{S_2}\|_2 \leq
\sqrt{\mathrm{rank}(A)}\|\hat A^{\lambda}-A\|_2
\end{equation}
and to Assumption 1, it follows from (\ref{extremum_B}) and
(\ref{M0}) that
\begin{eqnarray}
\label{extremum_C}  \nonumber &&\|{\hat
A}^\lambda-A_0\|_{L_2(\Pi)}^2+ \|{\hat A}^\lambda-A\|_{L_2(\Pi)}^2+
\lambda \|P_{S_1^{\perp}}{\hat A}^\lambda P_{S_2^{\perp}}\|_1 \\
&& \quad \leq \nonumber \|A-A_0\|_{L_2(\Pi)}^2 +\mu(\lambda
\sqrt{{\rm rank}(A)}+ \Lambda)\|{\hat A}^\lambda-A\|_{L_2(\Pi)}
\nonumber
\\
&& \quad \ \ \  + \ \Gamma\|P_{S_1^{\perp}}{\hat A}^\lambda
P_{S_2^{\perp}}\|_1.
\end{eqnarray}
Using the above bounds on $\Lambda$ and $\Gamma ,$ we obtain from
(\ref{extremum_C}) that
\begin{eqnarray*}
\label{extremum_C1}  \nonumber &&\|{\hat
A}^\lambda-A_0\|_{L_2(\Pi)}^2+ \|{\hat A}^\lambda-A\|_{L_2(\Pi)}^2+
(\lambda-2\Delta) \|P_{S_1^{\perp}}{\hat A}^\lambda P_{S_2^{\perp}}\|_1 \\
&& \quad \leq \|A-A_0\|_{L_2(\Pi)}^2 + (1+\sqrt{2})\mu \lambda
\sqrt{{\rm rank}(A)}\|{\hat A}^\lambda-A\|_{L_2(\Pi)}
\end{eqnarray*}
which implies
\begin{eqnarray*}
\label{extremum_C2}  \nonumber &&\|{\hat
A}^\lambda-A_0\|_{L_2(\Pi)}^2+ (\lambda-2\Delta) \|P_{S_1^{\perp}}{\hat A}^\lambda P_{S_2^{\perp}}\|_1 \\
&& \quad \leq \|A-A_0\|_{L_2(\Pi)}^2 +
\frac{1}{4}(1+\sqrt{2})^2\mu^2 \lambda^2 {\rm rank}(A).
\end{eqnarray*}
\end{proof}

The following immediate corollary of Theorem \ref{th:main} provides
a bound for the Frobenius error.
\begin{corollary}
Let $\mathbb A$ be a convex subset of $m_1\times m_2$ matrices
containing $A_0$, and let Assumption \ref{design_assumption} be
satisfied. If $\lambda \geq 2\|{\bf M}\|_{\infty},$ then
\begin{equation}
\label{first1} \|{\hat A}^\lambda-A_0\|_2^2 \leq  \lambda\mu^2
\min\left\{2 \|A_0\|_1, \
\left(\frac{1+\sqrt{2}}{2}\right)^2\lambda\mu^2 \, {\rm rank}(A_0)
\right\}\,.
\end{equation}
\end{corollary}

%{\sc Remark~1.} %\label{rem1}
%A small modification of the proof of Theorem \ref{th:main} leads to
%more general oracle inequalities than (\ref{third}) and
%(\ref{second}) where Assumption 1 is not needed. Indeed, consider a
%linear subspace $\tilde{\mathbb A}$ of the set of all $m_1\times
%m_2$ matrices, and assume that ${\mathbb A}$ is a convex subset of
%$\tilde{\mathbb A}$. For any matrix $A\in {\mathbb R}^{m_1\times
%m_2}$ with support $(S_1,S_2)$, define
%$$
%\mu_*(A)= \inf\Bigl\{\mu'>0: \|P_{S_1} B\|_2 + \|B P_{S_2} \|_2 \leq \mu'
%\|B\|_{L_2(\Pi)}, \,\forall B \in \tilde{\mathbb A}\Bigr\}.
%$$
%Then, without any assumption on $A$, inequalities (\ref{third}) and
%(\ref{second}) remain valid with $\mu$ replaced by $\mu_*(A)$. The
%proof of this fact differs from that of Theorem \ref{th:main} only
%in that we now use the bounds $\|{\cal P}_A({\hat
%A}^\lambda-A)\|_2\le \|P_{S_1} ({\hat A}^\lambda-A)\|_2 + \|({\hat
%A}^\lambda-A) P_{S_2} \|_2 \le \mu_*(A) \|{\hat
%A}^\lambda-A\|_{L_2(\Pi)}$ in (\ref{M0}) and the bound
%$\sqrt{\mathrm{rank}(A)}\|P_{S_1}(\hat A^{\lambda}-A)P_{S_2}\|_2
%\leq \mu_*(A)\sqrt{\mathrm{rank}(A)}\|\hat
%A^{\lambda}-A\|_{L_2(\Pi)}$ in (\ref{vnutr}).
%
%\bigskip

%The quantity $\mu_*(A)$ introduced in Remark~1 can be rather large
%for reasonable examples of $\Pi$ since the inequality $\|P_{S_1}
%B\|_2 + \|B P_{S_2} \|_2 \leq \mu' \|B\|_{L_2(\Pi)}$ in its
%definition should be satisfied for all $B$ in the linear space
%$\tilde{\mathbb A}$.
Next, we consider a version of Theorem \ref{th:main} under weaker
assumptions which are akin to Restricted Eigenvalue condition in
sparse estimation of vectors. For simplicity, we will do it only
when the domain ${\mathbb A}$ of minimization in (\ref{ERM}) is a
linear subspace of ${\mathbb R}^{m_1\times m_2}.$ Recall that, given
$A\in {\mathbb A}$ with support $(S_1,S_2),$ we denote
$$
{\cal P}_A (B):=B-P_{S_1^{\perp}}BP_{S_2^{\perp}},\ {\cal
P}_A^{\perp}(B):=P_{S_1^{\perp}}BP_{S_2^{\perp}}, \ B\in {\mathbb
R}^{m_1\times m_2},
$$
and, for $c_0\geq 0,$ define the following cone of matrices:
$$
{\mathbb C}_{A,c_0}:=\Bigl\{B\in {\mathbb A}: \|{\cal
P}_A^{\perp}(B)\|_1\leq c_0 \|{\cal P}_A(B)\|_1\Bigr\}.
$$
Finally, define
$$
\mu_{c_0}(A):=\inf\Bigl\{\mu'>0: \|{\cal P}_A (B)\|_2 \leq \mu'
\|B\|_{L_2(\Pi)}, \,\forall B\in {\mathbb C}_{A,c_0}\Bigr\}.
$$
Note that $\mu_{c_0}(A)$ is a nondecreasing function of $c_0 .$ For
$c_0=+\infty,$ the quantity $\mu_{\infty}(A)$ has a simple meaning:
it is equal to the norm of the linear transformation $B\mapsto {\cal
P}_A(B)$ from the space ${\mathbb A}$ equipped with the
$L_2(\Pi)$-norm into the space of all matrices equipped with the
Frobenius norm. For $c_0=0,$ $\mu_0(A)$ is the norm of the same
linear transformation restricted to the subspace of ${\mathbb A}$
consisting of all matrices $B\in {\mathbb A}$ with ${\cal
P}_A^{\perp}(B)=0.$ We are more interested in the intermediate
values, $c_0 \in (0,+\infty).$ In this case, $\mu_{c_0}(A)$ is the
``norm'' of the linear mapping ${\cal P}_A$ restricted to the cone
of matrices $B$ for which ${\cal P}_A(B)$ is the dominant part and
${\cal P}_A^{\perp}(B)$ is ``small''. Note that the rank of ${\cal
P}_A(B)$ is not larger than $2{\rm rank}(A),$ so, when the rank of
$A$ is small, the matrices in ${\mathbb C}_{A,c_0}$ are
approximately ``low-rank''. The quantities of the same flavor have
been previously used in the literature on Lasso, Dantzig selector
and other methods of sparse estimation of vectors. In these
problems, they can be expressed in terms of ``restricted
eigenvalues" of certain Gram matrices, cf. the Restricted Eigenvalue
condition in \cite{BRT09} for the fixed design case and similar
distribution dependent conditions in \cite{kol_dantz} for the random
design case. Such conditions are also considered in \cite{negaban10}
for the matrix case. In what follows, we use the value $c_0=5$ and
set $ \mu(A):=\mu_5(A). $

\begin{theorem}\label{th:main_RE}
Let ${\mathbb A}$ be a linear subspace of ${\mathbb R}^{m_1\times
m_2}.$ If $\lambda \ge 3\|{\bf M}\|_\infty$, then
\begin{equation}
\label{eq:th:main_RE} \|{\hat A}^\lambda-A_0\|_{L_2(\Pi)}^2 \leq
\inf_{A\in {\mathbb A}}\Bigl[\|A-A_0\|_{L_2(\Pi)}^2 + \lambda^2
\mu^2(A){\rm rank}(A)\Bigr]\,.
\end{equation}
\end{theorem}

\begin{proof}
Fix $A\in\mathbb A$ with support $(S_1,S_2).$ If $\langle {\hat
A}^\lambda-A_0, {\hat A}^\lambda-A\rangle_{L_2(\Pi)}\le 0$, then we
trivially have $\|{\hat A}^\lambda-A_0\|_{L_2(\Pi)}^2 \leq
\|A-A_0\|_{L_2(\Pi)}^2$ in view of (\ref{identity}). Thus, assume
that $\langle {\hat A}^\lambda-A_0, {\hat
A}^\lambda-A\rangle_{L_2(\Pi)}> 0$. In this case,
(\ref{extremum_A'}) and an obvious modification of (\ref{uv}) imply
\begin{equation}\label{extremum_B1}
\lambda \|P_{S_1^{\perp}} {\hat A}^\lambda P_{S_2^{\perp}}\|_1 \leq
\lambda \|{\cal P}_A({\hat A}^\lambda-A)\|_1 +2 \langle {\bf M},
{\hat A}^\lambda-A \rangle.
\end{equation}
Now,
\begin{eqnarray}\label{extremum_B2} &&\langle {\bf
M}, {\hat A}^\lambda-A \rangle = \langle {\bf M}, {\cal P}_A(\hat
A^{\lambda}-A)\rangle + \langle {\bf M}, {\cal P}_{A}^{\perp}({\hat
A}^\lambda-A)\rangle
\nonumber\\
&&\quad \leq \|{\bf M}\|_{\infty}\left( \|{\cal P}_A({\hat
A}^\lambda-A)\|_1 +\|{\cal P}_A^{\perp}({\hat
A}^\lambda-A)\|_1\right)\,.
\end{eqnarray}
By (\ref{extremum_B1}) and (\ref{extremum_B2}),
\begin{equation}\label{222}
(\lambda - 2 \Delta) \|{\cal P}_A^{\perp}(\hat A^{\lambda}-A)\|_1
\leq (\lambda + 2\Delta) \|{\cal P}_A({\hat A}^\lambda-A)\|_1.
\end{equation}

For $\lambda \ge 3\Delta$, this yields
$$
\|{\cal P}_A^{\perp}({\hat A}^\lambda-A)\|_1 \leq 5 \|{\cal
P}_A({\hat A}^\lambda-A)\|_1,
$$
which implies that ${\hat A}^\lambda-A\in {\mathbb C}_{A,5},$ and
thus $\|{\cal P}_A({\hat A}^\lambda-A)\|_2\le \mu (A)\|{\hat
A}^\lambda-A\|_{L_2(\Pi)}$. Combining this inequality with
(\ref{extremum_B}), (\ref{M0}), (\ref{2.10a}), (\ref{2.10b}) and
using that $\lambda \ge 3\Delta$, after some algebra we get
\begin{eqnarray*}
&& \|{\hat A}^\lambda-A_0\|_{L_2(\Pi)}^2+ \|{\hat
A}^\lambda-A\|_{L_2(\Pi)}^2 +(\lambda/3) \|P_{S_1^{\perp}} {\hat
A}^\lambda P_{S_2^{\perp}}\|_1 \nonumber
\\
&&\quad \leq \|A-A_0\|_{L_2(\Pi)}^2 +(1+2\sqrt{2}/3)\mu(A)\lambda
\sqrt{{\rm rank}(A)} \, \|{\hat A}^\lambda-A\|_{L_2(\Pi)}
\\
&&\quad \leq  \|A-A_0\|_{L_2(\Pi)}^2 + \|{\hat
A}^\lambda-A\|_{L_2(\Pi)}^2 + \mu^2(A)\lambda^2 {\rm rank}(A).
\end{eqnarray*}
\end{proof}

%{\bf Remark}. %Note that the result of Theorem \ref{th:main} is also true in a more general
%case when the target matrix $A_0$ and the design matrices $X_j$ belong to a
%known linear subspace $L$ of the space of $m_1\times m_2$ matrices. In this
%case, it is natural to define a version of estimator $\hat A^{\lambda}$ by minimizing
%the functional $L_n(A)$ over the subspace $L.$  The inequalities of Theorem \ref{th:main}
%still hold with the infimum now taken over all $A\in L.$ This observation can be used
%to obtain a variety of oracle inequalities for special classes of matrices.
As a simple example, consider the case when $m_1=m_2$, $\mathbb A$
is the space of all diagonal matrices, and $X_i$ also belong to
$\mathbb A$. Then the trace regression model (\ref{tracereg1})
becomes the usual linear regression model. The Schatten $p$-norms
are in this case equivalent to the $\ell_p$-norms with the operator
norm $\|\cdot\|_{\infty}$ being the $\ell_{\infty}$-norm and the
rank of matrix $A$ characterizing the sparsity of the corresponding
vector. The problem of minimizing the functional $L_n(A)$ over the
space $\mathbb A$ is a Lasso-type penalized empirical risk
minimization. In particular, it coincides with the standard Lasso if
all $X_i$ are non-random. Inequalities of Theorem \ref{th:main} and
(\ref{eq:th:main_RE}) become, in this case, sparsity oracle
inequalities for the Lasso-type estimators. It is noteworthy that
these inequalities are sharp (i.e., with leading constant $1$),
which was not achieved in the past work. The random matrix ${\bf M}$
is also diagonal and its norm $\|{\bf M}\|_{\infty}$ is just the
$\ell_{\infty}$-norm of the corresponding random vector, which is
the sum of independent random vectors. Hence, it is easy to provide
probabilistic bounds on $\|{\bf M}\|_{\infty}$ using, for instance,
the classical Bernstein inequality and the union bound. We give an
example of such an application of Theorem \ref{th:main_RE} in
Section \ref{subsec:lasso}.

\section{Upper bounds for matrix completion}\label{sec:matrix_compl}

In this section we consider implications of the general oracle
inequalities of Theorem \ref{th:main} for the model of USR matrix
completion. Thus, we assume that the matrices $X_i$ are i.i.d.
uniformly distributed in the matrix completion basis $\mathcal X,$
which implies that $\|A\|_{L_2(\Pi)}^2 = (m_1m_2)^{-1}\|A\|_2^2$ for
all matrices $A\in\R^{m_1\times m_2}$, and we set
$\mu=\sqrt{m_1m_2}$. The estimator ${\hat A}^\lambda$ is then
defined by (here and further on we set ${\mathbb A}={\mathbb
R}^{m_1\times m_2}$ in the case of matrix completion):
\begin{eqnarray}\nonumber
{\hat A}^\lambda &=&{\rm argmin}_{A\in {\mathbb R}^{m_1\times m_2}}
\left(\frac1{m_1m_2}\|A\|_2^2 -\biggl\langle \frac{2}{n}\sum_{i=1}^n
Y_i X_i, A \biggr\rangle + \lambda
\|A\|_1\right)\\
&=&{\rm argmin}_{A\in {\mathbb R}^{m_1\times m_2}} \Big(
\|A-{\mathbf X}\|_2^2 + \lambda m_1m_2\|A\|_1\Big),\label{ERM1}
\end{eqnarray}
where
$$
{\mathbf X}= \frac{m_1m_2}{n}\sum_{i=1}^n Y_i X_i.
$$
We can also write ${\hat A}^\lambda$ explicitly:
\begin{equation}\label{alambda}
{\hat A}^\lambda = \sum_{j}^{} \big(\sigma_j({\mathbf X})-\lambda
m_1m_2/2\big)_+ u_j({\mathbf X}) v_j({\mathbf X})^\top
\end{equation}
where $x_+=\max\{x,0\}$, $\sigma_j({\mathbf X})$ are the singular
values and $u_j({\mathbf X}), v_j({\mathbf X})$ are the left and
right singular vectors of ${\mathbf X}=\sum_{j=1}^{{\rm
rank}({\mathbf X})} \sigma_j({\mathbf X}) u_j({\mathbf X})
v_j({\mathbf X})^\top$. Thus, ${\hat A}^\lambda$ has a particularly
simple form; it is obtained by soft thresholding of singular values
in the SVD of ${\mathbf X}$. To see why (\ref{alambda}) gives the
solution of (\ref{ERM1}), note that, in view of (\ref{subdiff}), the
subdifferential of $F(A)=\|A-{\mathbf X}\|_2^2 + \lambda
m_1m_2\|A\|_1$ is the set of matrices
$$
\partial F(A)=\Bigl\{2(A-{\mathbf X}) + \lambda m_1m_2\left(\sum_{j=1}^r u_j v_j^\top +
P_{S_1^{\perp}}W P_{S_2^{\perp}}\right):\ \|W\|_\infty\leq 1
\Bigr\}\,,
$$
where $r, u_j, v_j, S_1, S_2$ correspond to the SVD of $A$. Since
$A\mapsto F(A)$ is strictly convex, the minimizer ${\hat A}^\lambda$
is unique, and the condition ${\bf 0}\in
\partial F({\hat A}^\lambda)$ is necessary and sufficient
characterization of the minimum, where ${\bf 0}$ is the zero
$m_1\times m_2$ matrix. Considering $$W=\sum_{j: \,\sigma_j({\mathbf
X})<\lambda m_1m_2/2} \Big(\frac{2\sigma_j({\mathbf X})}{\lambda
m_1m_2}- 1\Big) u_j({\mathbf X}) v_j({\mathbf X})^\top,$$ it is easy
to check that (\ref{alambda}) satisfies this condition.

We will see that the soft thresholding representation
(\ref{alambda}) helps to understand in an easy way some theoretical
properties of ${\hat A}^\lambda$. However, it may not be always
preferable for computational issues. Indeed, the standard techniques
of computation of the SVD can become numerically instable when the
dimension is high. On the other hand, we can always compute ${\hat
A}^\lambda$ from (\ref{ERM1}) using the methods of convex
programming free from this drawback.

In view of Theorem \ref{th:main}, to get the oracle inequalities in
a closed form it remains only to specify the value of regularization
parameter $\lambda$ such that $\lambda \geq 2\|{\bf M}\|_{\infty}$
with high probability. This requires some assumptions on the
distribution of $(X_i,Y_i)$, and the value of $\lambda$ will be
different under different assumptions. We will consider only the
following two cases of particular interest.
\begin{itemize}
\item {\it Sub-exponential noise and matrices with uniformly bounded
entries.} There exist constants $\sigma,c_1>0$, $\alpha\ge 1$ and
$\tilde c$ such that
\begin{equation}\label{subexp}
\max_{i=1,\dots,n}\E \exp
\left(\frac{|\xi_i|^\alpha}{\sigma^\alpha}\right) < \tilde c, \quad
\mathbb{E}\xi_i^2 \geq c_1 \sigma^2,\,\forall 1\leq i \leq n,
\end{equation}
and $\max_{i,j}|a_0(i,j)|\leq a$ for some constant $a$.
\item {\it Statistical learning setting.} There exists a constant
$\eta$ such that $\max_{i=1,\dots,n}|Y_i|\le \eta$ almost surely.
\end{itemize}
In both cases, we obtain the upper bounds for $\|{\bf M}\|_{\infty}$
(that we call the {\it stochastic error}) using the non-commutative
Bernstein inequalities, cf. Section \ref{sto}. The resulting values
of $\lambda$ and the corresponding oracle inequalities are given in
the next two theorems.

Set $m=m_1+m_2$. In what follows, we will denote by $C$ absolute
positive constants, possibly different on different occasions.

\begin{theorem}\label{th:completion_gaussian} Let $X_i$ be i.i.d.
uniformly distributed on $\mathcal X$, and the pairs $(X_i,Y_i)$ be
i.i.d.
 Assume that $\max_{i,j}|a_0(i,j)|\leq a$ for some
constant $a$, and that condition (\ref{subexp}) holds.  For $t>0,$
consider the regularization parameter $\lambda$ satisfying
\begin{equation}\label{eq:th:completion_gaussian}
\lambda \geq C(\sigma \vee a)\max\left\{\
\sqrt{\frac{t+\log(m)}{(m_1\wedge m_2)n}}\, , \
\frac{(t+\log(m))\log^{1/\alpha}(m_1\wedge m_2)}{n} \right\},
\end{equation}
where $C>0$ is a large enough constant that can depend only on
$\alpha,c_1,\tilde c$. Then with probability at least $1-3e^{-t}$ we
have
\begin{equation}\label{ora_ineq_completion}
\|{\hat A}^\lambda - A_0\|_{2}^2 \leq \|A - A_0\|_{2}^2 +
m_1m_2 \,\min\bigg\{2 \lambda  \|A\|_1, \
\left(\frac{1+\sqrt{2}}{2}\right)^2 m_1m_2 \lambda^2 \, {\rm
rank}(A) \bigg\}
\end{equation}
for all $A\in \mathbb R^{m_1\times m_2}$.
\end{theorem}

\begin{theorem}\label{th:completion_bounded}
Let $X_i$ be i.i.d. uniformly distributed on $\mathcal X$. Assume
that $\max_{i=1,\dots,n}|Y_i|\le \eta$ almost surely for some
constant $\eta$. For $t>0$ consider the regularization parameter
$\lambda$ satisfying
\begin{equation}\label{eq:th:completion_bounded}
\lambda \geq 4 \eta\,\max\left\{  \sqrt{\frac{t+\log(m)}{(m_1\wedge
m_2)n}}\,, \ \frac{2(t+\log(m))}{n} \right\}\,.
\end{equation}
Then with probability at least $1-e^{-t}$ inequality
(\ref{ora_ineq_completion}) holds for
 all $A\in \mathbb R^{m_1\times m_2}$.
\end{theorem}

Theorems \ref{th:completion_gaussian} and
\ref{th:completion_bounded} follow immediately from Theorem
\ref{th:main} and Lemmas \ref{lem:completion_bounded},
\ref{lem:tau1} and \ref{lem:tau2} with $\mu=\sqrt{m_1m_2}$.

Note that the natural choice of $t$ in Theorems
\ref{th:completion_gaussian} and \ref{th:completion_bounded} is of
the order $\log(m)$, since a larger $t$ leads to slower rate of
convergence and a smaller $t$ does not improve the rate but makes
the concentration probability smaller. Note also that, under this
choice of $t$, the second terms under the maxima in
(\ref{eq:th:completion_gaussian}) and
(\ref{eq:th:completion_bounded}) are negligible for the values of
$n,m_1,m_2$ such that the term containing ${\rm rank}(A_0)$ in
(\ref{ora_ineq_completion}) is meaningful. Indeed, if $t$ is of the
order $\log(m)$, the condition that $m_1m_2\lambda^2\ll 1$
necessarily implies $n\gg (m_1 \vee m_2) \log(m)$. On the other
hand, the negligibility of the second terms under the maxima in
(\ref{eq:th:completion_gaussian}) and
(\ref{eq:th:completion_bounded}) is approximately equivalent to $n>
(m_1 \wedge m_2)\log^{1+2/\alpha}(m)$ and $n>(m_1 \wedge m_2)
\log(m)$ respectively. Based on these remarks, we can choose
$\lambda$ in the form
\begin{equation}\label{lam}
\lambda = C_* c_* \sqrt{\frac{\log(m)}{(m_1\wedge m_2)n}}\ ,
\end{equation}
where $c_*$ equals either $\sigma \vee a$ or $\eta$ and the constant
$C_*>0$ is large enough, and we can state the following corollary
that will be further useful for minimax considerations. Define
$\tau>0$ by
$$
\tau^2 = \left(\frac{1+\sqrt{2}}{2}\right)^2C_*^2c_*^2
\frac{M\,\log(m)}{n} \,,
$$
where $M=\max(m_1,m_2)$, and $m=m_1+m_2$.
\begin{corollary}\label{cor:completion} Let one of the
sets of conditions (i) or (ii) below be satisfied:

 (i) The assumptions
of Theorem \ref{th:completion_gaussian} with $\lambda$ as in
(\ref{lam}), $n>(m_1 \wedge m_2) \log^{1+2/\alpha}(m)$, $c_*=\sigma
\vee a$, and a large enough constant $C_*>0$ that can depend only on
$\alpha,c_1,\tilde c$.

(ii) The assumptions of Theorem \ref{th:completion_bounded} with
$n>4(m_1 \wedge m_2) \log(m)$, $\lambda$ as in (\ref{lam}),
$c_*=\eta$, and $C_*=4$.

Then, with probability at least $1-3/(m_1+m_2)$,
\begin{equation}
\label{first2_1} \frac1{m_1m_2}\|{\hat A}^\lambda-A_0\|_2^2 \leq
\min_{A\in \R^{m_1\times m_2}} \Big(\frac1{m_1m_2}\|A-A_0\|_2^2 +
\tau^2{\rm rank}(A)\Big) \,,
\end{equation}
and, in particular,
\begin{equation} \label{first2} \frac1{m_1m_2}\|{\hat
A}^\lambda-A_0\|_2^2 \leq
\left(\frac{1+\sqrt{2}}{2}\right)^2C_*^2c_*^2\log(m) \frac{M\,{\rm
rank}(A_0)}{n}  \,,
\end{equation}
where $M=\max(m_1,m_2)$, and $m=m_1+m_2$. Furthermore, with the same
probability,
\begin{equation}
\label{first2_2} \frac1{m_1m_2}\|{\hat A}^\lambda-A_0\|_2^2 \leq
\sum_{j=1}^{{\rm rank}(A_0)} \min\bigg\{\tau^2, \
\frac{\sigma_j^2(A_0)}{m_1m_2}\bigg\}\le\inf_{0<q\le 2}
\,\frac{\tau^{2-q}\|A_0\|_q^q}{(m_1m_2)^{q/2}}\,.
\end{equation}
\end{corollary}
\begin{proof}
Inequalities (\ref{first2_1}) and (\ref{first2}) are straightforward
in view of Theorems \ref{th:completion_gaussian} and
\ref{th:completion_bounded}. To prove (\ref{first2_2}) it suffices
to note that, for any $\kappa>0$, $0<q\le 2$,
\begin{eqnarray*}
\min_{A} \Big(\|A-A_0\|_2^2 + \kappa^2\, {\rm rank}(A)\Big)&=&
\sum_{j}^{} \min\{\kappa^2, \sigma_j^2(A_0)\}= \kappa^2 \sum_{j}^{}
\min\bigg\{1, \bigg(\frac{\sigma_j(A_0)}{\kappa}\bigg)^2\bigg\}\\
 &\le & \kappa^2 \sum_{j}^{}
\min\bigg\{1, \bigg(\frac{\sigma_j(A_0)}{\kappa}\bigg)^q\bigg\} \le
\kappa^{2-q} \|A_0\|_q^q.
\end{eqnarray*}

\end{proof}

Inequality (\ref{first2}) guarantees that the normalized Frobenius
error $(m_1m_2)^{-1}\|{\hat A}^\lambda-A_0\|_2^2$ of the estimator
${\hat A}^\lambda$ is small whenever $n> C (m_1\vee m_2)\log(m){\rm
rank}(A_0)$ with a large enough $C>0$. This quantifies the sample
size $n$ necessary for successful matrix completion from noisy data.

Note that we can choose $\lambda$ not necessarily equal but also
greater or equal to the right hand side of (\ref{lam}), or
equivalently, $\lambda= tC_* c_* \sqrt{\frac{\log(m)}{(m_1\wedge
m_2)n}}$ for any $t\ge 1$. Then the resulting oracle inequalities
will remain of the same form with $\tau^2$ multiplied by the
constant $t^2$.

Keshavan et al. \cite{Keshavan}, Theorem 1.1,  under a sampling
scheme different from ours (sampling without replacement) and
sub-gaussian errors, proposed an estimator ${\hat A}$ satisfying,
with probability at least $1-(m_1\wedge m_2)^{-3}$,
\begin{equation}\label{SOI-rank}
\frac1{m_1 m_2}\|\hat{A} - A_0\|_{2}^2 \leqslant C \sqrt{\beta}
  \,\log (n) \,\frac{ M \,{\rm rank}(A_0)}{ n}\,,
\end{equation}
where $C>0$ is a constant, and $\beta=(m_1\vee m_2)/(m_1\wedge m_2)$
is the aspect ratio. A drawback is that the construction of ${\hat
A}$ in \cite{Keshavan} requires the exact knowledge of ${\rm
rank}(A_0)$ (although it does not seem to require the knowledge of
$a$). Furthermore, the bound (\ref{SOI-rank}) is suboptimal for
"very rectangular" matrices, i.e., when~$\beta\gg1$. Candes and Plan
\cite{cp09} provide a coarser bound than (\ref{SOI-rank}), not
guaranteeing a simple consistency when $n\to\infty$ whatever are $M$
and ${\rm rank}(A_0)$ (see \cite{nw10} for more detailed comments on
\cite{cp09}).

%As a minor point, the logarithmic
%factor $\log(m)$ in (\ref{first2}) is smaller than $\log (n)$
%appearing in (\ref{SOI-rank}) for values $n,m_1,m_2$ such that the
%bounds are meaningful.

\section{Lower Bounds}\label{lower bounds}

In this section, we prove the minimax lower bounds showing that the
rates attained by our estimator are optimal up to logarithmic
factors. The argument here is close to \cite{Rohde} where the lower
bounds are obtained on the Schatten balls. However, we consider
different classes that consist of matrices with uniformly (in $m_1,
m_2, n$) bounded entries. We cannot apply directly the lower bounds
of Theorem 6 in \cite{Rohde} for USR matrix completion on the
Schatten balls because they are achieved on matrices with entries,
which are not uniformly bounded for $m_1m_2\gg n$.

We will need the following assumption, which is similar in spirit
but, in general, substantially weaker than the usual Restricted
Isometry condition.
\begin{assumption}\label{RI} {\bf(Restricted Isometry in Expectation.)} For some $1\leq r \leq \min(m_1,m_2)$ and some $0<\mu<\infty$
that there exists a constant $\delta_r\in [0,1)$ such that
$$
(1-\delta_r)\|A\|_2 \leq \mu\|A\|_{L_2(\Pi)} \leq
(1+\delta_r)\|A\|_2,
$$
for all matrices $A\in \R^{m_1\times m_2}$ with rank at most $r$.
\end{assumption}

For the particular case of fixed $X_i$ (cf. Example 4 in the
Introduction), Assumption~\ref{RI} coincides with the matrix version
of scaled restricted isometry with scaling factor $\mu$
\cite{Rohde}.

{\sc Remark~1.} Inspection of the proof of Theorem~\ref{th:lower}
shows that it remains valid if we replace $1-\delta_r$ and
$1+\delta_r$ by arbitrary positive constants $\nu_1$ and $\nu_2$
such that $\nu_1\le\nu_2$. We use the formulation involving
$\delta_r$ only to ease parallels to the usual restricted isometry
condition.

We will denote by $\inf_{\hat{A}}$ the infimum over all estimators
$\hat{A}$ with values in $\R^{m_1\times m_2}$. For any integer $r\le
\min(m_1,m_2)$ and any $a>0$ we consider the class of matrices
\begin{equation*}\label{Alb}
{\cal A}(r,a)= \big\{A_0\in\,\R^{m_1\times m_2}:\,
\mathrm{rank}(A_0)\leq r,\, \max_{i,j}|a_0(i,j)|\,\leq\, a\big\}\,.
\end{equation*}
For any $A\in\R^{m_1\times m_2}$, let $\P_A$ denote the probability
distribution of the observations $(X_1,Y_1,\dots,X_n,Y_n)$ with $\E
(Y_i|X_i)=\langle A,X_i\rangle$. We set for brevity
$M=\max(m_1,m_2)$.

\begin{theorem}\label{th:lower} Fix $a>0$ and an integer
$1\leq  r\leq \min(m_1,m_2)$. Let Assumption \ref{RI} be satisfied
with some $\mu>0$. Assume that $\mu^2 r\leq n \min(m_1,m_2)$, and
that conditionally on $X_i$, the variables $\xi_i$ are Gaussian
${\cal N}(0,\sigma^2)$, $\sigma^2>0$, for $i=1,\dots,n$.  Then there
exist absolute constants $\beta\in(0,1)$ and $c>0$, such that
\begin{equation}\label{eq:lower1}
\inf_{\hat{A}}%\hspace{-2mm}
\sup_{\substack{A_0\in\,{\cal A}(r,a)
}}%\hspace{-2mm}
\P_{A_0}\bigg(\|\hat{A}-A_0\|^2_{L_2(\Pi)}> c(1-\delta_r)^2(\sigma
\wedge a)^2 \frac{Mr}{n} \bigg)\ \geq\ \beta.
\end{equation}
\end{theorem}
\begin{proof}
Without loss of generality, assume that $M=\max(m_1,m_2)=m_1\geq
m_2$. For some constant $0\leq \gamma \leq 1$ we define
$$
\mathcal{C} \, =\Big\{ \tilde{A}=(a_{ij})\in\R^{m_1\times r}:
a_{ij}\in\Big\{0, \gamma(\sigma \wedge a) \Big(\frac{\mu^2
r}{m_2n}\Big)^{1/2}\Big\}\,, \forall 1\leq i \leq m_1,\, 1\leq j\leq
r\Big\},
$$
and consider the associated set of block matrices
$$
\mathcal{B}(\mathcal{C})\ =\ \Big\{
A=(\begin{array}{c|c|c|c}\tilde{A}&\cdots&\tilde{A}&O
\end{array})\in\R^{m_1\times m_2}: \tilde{A}\in \mathcal{C}\Big\},
$$
where $O$ denotes the $m_1\times (m_2-r\lfloor m_2/r \rfloor )$ zero
matrix, and $\lfloor x \rfloor$ is the integer part of $x$.

By construction, any element of $\mathcal{B}(\mathcal{C})$ as well
as the difference of any two elements of $\mathcal{B}(\mathcal{C})$
has rank at most $r$ and the entries of any matrix in
$\mathcal{B}(\mathcal{C})$ take values in $[0,a]$. Thus,
$\mathcal{B}(\mathcal{C})\subset {\cal A}(r,a)$. Due to the
Varshamov-Gilbert bound (cf. Lemma 2.9 in \cite{tsy_09}), there
exists a subset $\AA^0\subset\mathcal{B}(\mathcal{C})$ with
cardinality $\mathrm{Card}(\AA^0) \geq 2^{rm_1/8}+1$ containing the
zero $m_1\times m_2$ matrix ${\bf 0}$ and such that, for any two
distinct elements $A_1$ and $A_2$ of $\AA^0$,
\begin{equation}\label{lower_2}
\Arrowvert A_1-A_2\Arrowvert_{2}^2 \geq \frac{m_1r}{8}
\left(\gamma^2(\sigma \wedge a)^2 \frac{\mu^2r}{m_2n} \right)
\left\lfloor \frac{m_2}{r}\right\rfloor \geq
\frac{\gamma^2}{16}(\sigma \wedge a)^2\frac{\mu^2m_1r}{n}\,.
\end{equation}
In view of Assumption \ref{RI}, this implies
\begin{align}\label{eq: condition A}
\Arrowvert A_1-A_2\Arrowvert_{L_2(\Pi)}^2\geq &
(1-\delta_r)^2\frac{\gamma^2}{16}(\sigma \wedge a)^2\frac{m_1r}{n}.
\end{align}
Using that, conditionally on $X_i$, the distributions of $\xi_i$ are
Gaussian, we get that, for any $A\in \AA_0$, the Kullback-Leibler
divergence $K\big(\P_{{\bf 0}},\P_{A}\big)$ between $\P_{{\bf 0}}$
and $\P_{A}$ satisfies
\begin{equation}\label{KLdiv}
K\big(\P_{{\bf 0}},\P_{A}\big)\ =\
\frac{n}{2\sigma^2}\|A\|_{L_2(\Pi)}^2 \leq
(1+\delta_r)^2\frac{\gamma^2}{2}m_1r.
\end{equation}
From (\ref{KLdiv}) we deduce that the condition
\begin{equation}\label{eq: condition C}
\frac{1}{\mathrm{Card}(\AA^0)-1} \sum_{A\in\AA^0}K(\P_{\bf
0},\P_{A})\ \leq\ \alpha \log \big(\mathrm{Card}(\AA^0)-1\big)
\end{equation}
is satisfied for any $\alpha>0$ if $ \gamma>0$ is chosen as a
sufficiently small numerical constant depending on $\alpha$. In view
of (\ref{eq: condition A}) and (\ref{eq: condition C}), the result
now follows by application of Theorem 2.5 in \cite{tsy_09}.
\end{proof}

In the USR matrix completion problem we have $\|A\|_{L_2(\Pi)}^2 =
(m_1m_2)^{-1}\|A\|_2^2$ for all matrices $A\in\R^{m_1\times m_2}$.
Thus, the corresponding lower bound follows immediately from the
previous theorem with $\delta_r=0$ and $\mu= \sqrt{m_1m_2}$.

\begin{theorem}\label{th:lower_completion} Fix $a>0$ and an integer
$r$ such that $1\leq  r\leq \min(m_1,m_2)$, $M r\leq n$. Let the
matrices $X_i$ be i.i.d. uniformly distributed on $\mathcal X$ and
let, conditionally on $X_i$, the variables $\xi_i$ be Gaussian
${\cal N}(0,\sigma^2)$, $\sigma^2>0$, for $i=1,\dots,n$. Then there
exist absolute constants $\beta\in(0,1)$ and $c>0$, such that
\begin{equation}\label{eq:lower2}
\inf_{\hat{A}}\ \sup_{\substack{A_0\in\,{\cal A}(r,a) }}
\P_{A_0}\bigg(\frac1{m_1m_2}\|\hat{A}-A_0\|^2_{2}> c(\sigma\wedge
a)^2 \frac{Mr}{n} \bigg)\ \geq\ \beta.
\end{equation}
\end{theorem}

Comparing Theorem~\ref{th:lower_completion} with
Corollary~\ref{cor:completion}(i) we see that, in the case of
Gaussian errors $\xi_i$, the rate of convergence of our estimator
$\hat A^\lambda$ given in (\ref{first2}) is optimal (up to a
logarithmic factor) in a minimax sense on the class of matrices
${\cal A}(r,a)$.

Similar conclusion can be obtained for the statistical learning
setting. Indeed, assume that the pairs $(X_i,Y_i)$ are i.i.d.
realizations of a random pair $(X,Y)$ with distribution $P_{XY}$
belonging to the class
$$
{\cal P}_{A_0,\eta} = \left\{ P_{XY}: \ X\sim \Pi_0, |Y|\le \eta
\,{\rm (a.s.)}, \ \E(Y|X)=\langle A_0,X\rangle \, \right\},
$$
where $\Pi_0$ is the uniform distribution on $\mathcal X$, $1\le r
\le \min(m_1,m_2)$ is an integer, and $\eta>0$.
\begin{theorem}\label{th:lower_bounded} Let $n,m_1,m_2,r$ be as in
Theorem~\ref{th:lower}. Let $(X_i,Y_i)$ be i.i.d. realizations of a
random pair $(X,Y)$ with distribution $P_{XY}$. Then there exist
absolute constants $\beta\in(0,1)$ and $c>0$, such that
\begin{equation}\label{eq:th:lower_bounded2}
\inf_{\hat{A}}\ \sup_{{\rm rank}(A_0)\le r}\ \sup_{P_{XY}\in\,{\cal
P}_{A_0,\eta}} \P \bigg(\frac1{m_1m_2}\|\hat{A}-A_0\|^2_{2}> c\eta^2
\frac{Mr}{n} \bigg)\ \geq\ \beta.
\end{equation}
\end{theorem}
\begin{proof} We act as in the proof of
Theorem~\ref{th:lower} with some modifications. Assuming that
$M=\max(m_1,m_2)=m_1\geq m_2$ and $0\leq \gamma \leq 1/2$ we define
the class of matrices
$$
\mathcal{C}'\, =\Big\{ \tilde{A}=(a_{ij})\in\R^{m_1\times r}:
a_{ij}\in\Big\{0, \gamma\eta \Big(\frac{\mu^2
r}{m_2n}\Big)^{1/2}\Big\}\,, \forall \, 1\leq i \leq m_1,\, 1\leq
j\leq r\Big\},
$$
and take its block extension $\mathcal{B}(\mathcal{C}')$. Consider
the joint distributions $P_{XY}$ such that $X\sim \Pi_0$ and,
conditionally on $X$, $Y=\eta$ with probability
$p_{A_0}(X)=1/2+\langle A_0,X\rangle/(2\eta)$ and $Y=-\eta$ with
probability $1-p_{A_0}(X)=1/2-\langle A_0,X\rangle/(2\eta)$, where
$A_0\in\mathcal{B}(\mathcal{C}')$. It is easy to see that such
distributions $P_{XY}$ belong to the class ${\cal P}_{A_0,\eta}$,
and our assumptions guarantee that $1/4 \le p_{A_0}(X)\le 3/4$,
${\rm rank}(A_0)\le r$ for all $A_0\in\mathcal{B}(\mathcal{C}')$. We
will denote the corresponding $n$-product measure by $\P_{A_0}$. For
any $A\in\mathcal{B}(\mathcal{C}')$, the Kullback-Leibler divergence
between $\P_{\bf 0}$ and $\P_{A}$ has the form
\begin{equation}\label{KLdiv1}
K\big(\P_{\bf 0},\P_{A}\big)\ =\ n \E\left(p_{\bf
0}(X)\log\frac{p_{\bf 0}(X)}{p_{A}(X)}+ (1-p_{\bf
0}(X))\log\frac{1-p_{\bf 0}(X)}{1-p_{A}(X)} \right)\,.
\end{equation}
Using the inequality $-\log (1+u)\le -u+u^2/2$, $\forall \ u> -1$,
and the fact that $1/4 \le p_{A}(X)\le 3/4$, we find that the
expression under the expectation in (\ref{KLdiv1}) is bounded by
$2(p_{\bf 0}(X)-p_{A}(X))^2$. This implies
$$
K\big(\P_{\bf 0},\P_{A}\big)\ \le
\frac{n}{2\eta^2}\|A\|_{L_2(\Pi_0)}^2.
$$
The remaining arguments are analogous to those in the proof of
Theorem~\ref{th:lower}.
\end{proof}

\section{Further results and examples} \label{sec:discussion}

\subsection{Recovery of the rank and specific lower bound}\label{sec:rank_recovery}

A notable property of the estimator ${\hat A}^\lambda$ in matrix
completion setting is that it has the same rank as the underlying
matrix $A_0$ with probability close to 1. As a consequence we can
establish a lower bound for the Frobenius error of $\hat A^\lambda$
with the rates matching up
to constants the upper bounds of Corollary 2. %, in the USR matrix completion,
\begin{theorem}\label{th:rank_recovery} Let $X_i$ be i.i.d.
uniformly distributed on $\mathcal X$ and let $\lambda$ satisfy the
inequality $\lambda \ge 2\|{\bf M}\|_\infty$ (as in Theorem
\ref{th:main}). Consider the estimator ${\hat A}^{\lambda'}$ with
$\lambda'=\lambda/(1-\delta)$ for some $0<\delta<1$. Set ${\hat r} =
{\rm rank}({\hat A}^{\lambda'})$. Then
\begin{equation}\label{eq:th:rank_recovery1}
{\hat r} \le {\rm rank}(A_0).
\end{equation}
If, in addition, $\displaystyle{\min_{j:\,\sigma_j(A_0)\ne0}}
\sigma_j(A_0)\ge \lambda' m_1m_2$, then
\begin{equation}\label{eq:th:rank_recovery2}
{\hat r} \ge {\rm rank}(A_0),
\end{equation}
and
\begin{equation}\label{eq:th:rank_recovery3}
\|\hat A ^{\lambda'} - A_0\|_2^2 \geq
\frac{\delta^2}{4(1-\delta)^2}{\rm rank}(A_0)\left(\lambda m_1
m_2\right)^2.
\end{equation}
\end{theorem}

\begin{proof} Note that ${\bf X} -A_0 = m_1m_2 {\bf M}$. Using
standard matrix perturbation argument (cf. \cite{SS90}, page 203),
we get, for all $j=1,\dots, m_1\wedge m_2$,
$$
|\sigma_j({\bf X})-\sigma_j(A_0)|\le \sigma_1({\bf X} -A_0) =m_1m_2
\|{\bf M}\|_\infty \le\frac{ \lambda m_1m_2}{2}= (1-\delta)\frac{
\lambda' m_1m_2}{2}\,.
$$
Since, by (\ref{alambda}), $\sigma_{\hat r}({\bf X})>\lambda'
m_1m_2/2$, we find that $\sigma_{\hat r}(A_0)>\delta\lambda'
m_1m_2/2$. This implies (\ref{eq:th:rank_recovery1}). Now, if
$\sigma_j(A_0)\ge \lambda' m_1m_2$ we get
$$
\sigma_j({\bf X})\ge \sigma_j(A_0)-|\sigma_j({\bf X})-\sigma_j(A_0)|
\ge \lambda' m_1m_2 - (1-\delta)\frac{ \lambda' m_1m_2}{2}>\frac{
\lambda' m_1m_2}{2}\,,
$$
and thus (\ref{eq:th:rank_recovery2}) follows.

To prove (\ref{eq:th:rank_recovery3}), denote by ${\cal
P}:\R^{m_1\times m_2}\to \R^{m_1\times m_2}$ the projector on the
linear span of matrices $\left(u_j(\mathbf X)v_j(\mathbf X)^\top,
j=1,\dots, r\right)$, where $r={\rm rank}(A_0)$. We have $\|\hat
A^{\lambda'} - A_0\|_2 \geq \|{\cal P}(\hat A^{\lambda'} - A_0)\|_2
\geq \|{\cal P}(\hat A^{\lambda'} - \mathbf X)\|_2 - \|{\cal
P}({\mathbf X} -A_0)\|_2$. Here $\|{\cal P}(\hat A^{\lambda'} -
\mathbf X)\|_2 = \sqrt{r}\lambda' m_1m_2/2$ in view of
(\ref{alambda}) and the fact that ${\hat r}=r$, cf.
(\ref{eq:th:rank_recovery1}) and (\ref{eq:th:rank_recovery2}). On
the other hand, $\|{\cal P}({\mathbf X} -A_0)\|_2\le \sqrt{r}\| {\bf
M}\|_{\infty}m_1m_2\le \sqrt{r}\lambda m_1m_2/2$. This implies
\begin{equation*}
\|\hat A^{\lambda'} - A_0\|_2 \geq  \sqrt{{\hat r}}
\left(\frac{\lambda' m_1 m_2 }{2} - (1-\delta) \frac{\lambda' m_1
m_2 }{2}\right) = \delta \sqrt{{\hat r}} \frac{\lambda' m_1 m_2
}{2}.
\end{equation*}

 \end{proof}

\begin{corollary}\label{cor:rank_recovery} Let the assumptions of Corollary
\ref{cor:completion} be satisfied. Consider the estimator ${\hat
A}^{\lambda'}$ with $$\lambda'=\frac{C_* c_*}{1-\delta}
\sqrt{\frac{\log(m)}{(m_1\wedge m_2)n}}
$$ for some
$0<\delta<1$. Set ${\hat r} = {\rm rank}({\hat A}^{\lambda'})$. Then
${\hat r} \le {\rm rank}(A_0)$ with probability at least
$1-3/(m_1+m_2)$. If, in addition,
\begin{equation}\label{nizhn}\displaystyle{\min_{j:\,\sigma_j(A_0)\ne0}} \sigma_j(A_0)\ge
\frac{C_* c_*}{1-\delta} \sqrt{m_1m_2}\sqrt{\frac{\log(m)(m_1\vee
m_2)}{n}},
\end{equation} then $ {\hat r} \ge {\rm rank}(A_0) $ and
\begin{equation}\label{specificlb}
\frac{1}{m_1m_2}\|\hat A ^{\lambda'} - A_0\|_2^2 \geq
\frac{\delta^2C_*^2c_*^2}{4(1-\delta)^2}{\rm rank}(A_0)\frac{ \log
(m)( m_1\vee m_2)}{n},
\end{equation}
with the same probability.
\end{corollary}
We note that the lower bound for $\sigma_j(A_0)$ in (\ref{nizhn}) is
not excessively high, since $\sqrt{m_1m_2}$ is a "typical" order of
the
largest singular value $\sigma_1(A_0)$ for non-lacunary matrices $A_0$. %if only few singular values are non-zero.
For example, if all the entries of $A_0$ are equal to some constant
$a$, the left hand side of (\ref{nizhn}) is equal to
$\sigma_1(A_0)=a\sqrt{m_1m_2}$.

\subsection{Risk bounds in statistical learning}\label{subsec:learning}

The results of the previous sections can be also extended to the
traditional statistical learning setting where $(X_i,Y_i)$ is a
sequence of i.i.d. replications of a random pair $(X,Y)$ with $X\in
\R^{m_1\times m_2}$ and $Y\in \R$, and there is no underlying model
determined by matrix $A_0$, i.e., we do not assume that
$\E(Y|X)=\langle A_0, X\rangle$. Then the above oracle inequalities
can be reformulated in terms of the prediction risk
$$
R(A)=\E \big[(Y-\langle A, X\rangle)^2\big], \quad \forall \ A\in
\R^{m_1\times m_2}.
$$
We illustrate this by an example dealing with USR matrix completion.
Specifically, Theorem~\ref{th:completion_bounded} is reformulated in
the following way.
\begin{theorem}\label{th:completion_learning}
Let $X_i$ be i.i.d. uniformly distributed on $\mathcal X$. Assume
that $|Y|\le \eta$ almost surely for some constant $\eta$. For $t>0$
consider the regularization parameter $\lambda$ satisfying
(\ref{eq:th:completion_bounded}). Then with probability at least
$1-e^{-t}$ we have
\begin{equation}\label{ora_ineq_completion_learning}
R({\hat A}^\lambda) \leq R(A) +
\min\bigg\{2 \lambda  \|A\|_1, \ \left(\frac{1+\sqrt{2}}{2}\right)^2
m_1m_2 \lambda^2 \, {\rm rank}(A) \bigg\}
\end{equation}
for all $A\in \mathbb R^{m_1\times m_2}$. In particular, under the
assumptions of Corollary~\ref{cor:completion}(ii),
\begin{equation}\label{ora_ineq_completion_learning1}
R({\hat A}^\lambda) \leq \min_{A\in \R^{m_1\times m_2}} \bigg( R(A)
+ 4 \left(1+\sqrt{2}\right)^2 \eta^2\log(m) \frac{M\,{\rm
rank}(A)}{n}\bigg)\,.
\end{equation}
\end{theorem}
This theorem can be also viewed as a result about the approximate
sparsity. We do not know whether the true underlying model is
described by some matrix $A_0,$ but we can guarantee that our
estimator is not far from the best approximation provided by
matrices $A$ with small rank or small nuclear norm.

Note that the results of Theorem~\ref{th:completion_learning} are
uniform over the class of distributions
$$
{\cal P}_{\eta} = \left\{ P_{XY}: \ X\sim \Pi_0, \ |Y|\le \eta
\,{\rm (a.s.)} \, \right\},
$$
where $\Pi_0$ is the uniform distribution on $\mathcal X$, and
$\eta>0$ is a constant. The corresponding lower bound is given in
the next theorem.
\begin{theorem}\label{th:lower_learning} Let $n,m_1,m_2,r$ be as in
Theorem~\ref{th:lower}. Let $(X_i,Y_i)$ be i.i.d. realizations of a
random pair $(X,Y)$ with distribution $P_{XY}$. Then
\begin{equation}\label{eq:th:lower_learning}
\inf_{\hat{A}}\ \sup_{{\rm rank}(A)\le r}\ \sup_{ P_{XY}\in\,{\cal
P}_{\eta}} \P \bigg(R({\hat A}) \geq R(A) + c\eta^2 \frac{Mr}{n}
\bigg)\ \geq\ \beta,
\end{equation}
where $\beta\in(0,1)$ and $c>0$ are absolute constants.
\end{theorem}
\begin{proof} For $\E(Y|X)=\langle A_0, X\rangle$ we have
$R(A)=\|A-A_0\|_{L_2(\Pi)}^2 + \sigma^2 = (m_1m_2)^{-1}\|A-A_0\|_2^2
+ \sigma^2$, where $\sigma^2=\E\big[(Y-\E(Y|X))^2\big]$. Thus, using
Theorem~\ref{th:lower_bounded} we get
\begin{eqnarray*}\label{eq:th:lower_learning1}
&&\sup_{{\rm rank}(A)\le r}\ \sup_{ P_{XY}\in\,{\cal P}_{\eta}} \P
\bigg(R({\hat A}) \geq R(A) + c\eta^2 \frac{Mr}{n} \bigg)\\
&& \geq \qquad \sup_{{\rm rank}(A)\le r}\ \sup_{P_{XY}\in\,{\cal
P}_{A,\eta}} \P \bigg(\frac1{m_1m_2}\|\hat{A}-A\|^2_{2}> c\eta^2
\frac{Mr}{n} \bigg) > \beta.
\end{eqnarray*}

\end{proof}

\noindent Inequalities (\ref{ora_ineq_completion_learning1}) and
(\ref{eq:th:lower_learning}) imply minimax rate optimality of ${\hat
A}^\lambda$ up to a logarithmic factor in the statistical learning
setting.

\subsection{Risks bounds in spectral norm}

The results of the previous sections on the Frobenius norm can be
extended to the spectral norm. In this subsection we consider the
USR matrix completion problem, i.e, we assume that the matrices
$X_i$ are i.i.d. uniformly distributed on $\mathcal X$, which
implies that $\|A\|_2^2 = (m_1m_2)^{-1}\|A\|_2^2$ for all matrices
$A\in \R^{m_1\times m_2}$.

\begin{theorem}\label{th:completion_spectral}
  Let $X_i$ be i.i.d. uniformly distributed on $\mathcal{X}$.
  Consider the estimator $\hat A^\lambda$ defined in (\ref{ERM1}).
  If $\lambda\geq \|\mathbf M\|_\infty$, then
  $$
\|\hat{A}^\lambda - A_0\|_\infty \leq \frac{3}{2}m_1m_2\lambda.
  $$
\end{theorem}

\begin{proof}
We have
\begin{equation*}
  \|\hat{A}^{\lambda} - A_0\|_\infty \leq \|\hat{A}^{\lambda} - \mathbf{X}\|_\infty +
  m_1m_2\|\mathbf M\|_\infty,
\end{equation*}
where we recall that $\mathbf X =
\frac{m_1m_2}{n}\sum_{i=1}^nY_iX_i$, $\mathbb E ({\bf X}) = A_0$ and
$\mathbf M$ is defined in (\ref{randM}). In view of (\ref{alambda}),
we clearly have $\|\hat{A}^{\lambda} - \mathbf{X}\|_\infty\leq
\lambda m_1m_2/2$. The result follows immediately since $\|\bf
M\|_\infty\leq \lambda$.
\end{proof}

As a consequence of the above theorem, we can derive the optimal
rate (up a to logarithmic factor) of USR matrix completion for the
spectral norm when the noise is sub-exponential or in the
statistical learning setting.

\begin{theorem}
Let one of the sets of conditions (i) or (ii) in Corollary
\ref{cor:completion} be satisfied. Then, with probability at least
$1-3/(m_1+m_2)$, we have
\begin{equation*}
\|\hat A^\lambda - A_0\|_\infty \leq C C_* c_*
\sqrt{m_1m_2}\sqrt{\frac{(m_1 \vee m_2) \log m}{n}},
\end{equation*}
where $C>0$ is an absolute constant.
\end{theorem}
\begin{proof}
The proof of this result is immediate by combining Theorem
\ref{th:completion_spectral} and Lemmas
\ref{lem:completion_bounded}, \ref{lem:tau1} and \ref{lem:tau2}.
\end{proof}

\begin{theorem}
(i) Let the conditions of Theorem \ref{th:lower_completion} be
satisfied. Then
\begin{equation}
\inf_{\hat{A}}%\hspace{-2mm}
\sup_{\substack{A_0\in\,{\cal A}(r,a)
}}%\hspace{-2mm}
\P_{A_0}\bigg(\|\hat{A}-A_0\|_\infty> c(\sigma \wedge
a)\sqrt{m_1m_2} \sqrt{\frac{m_1 \vee m_2}{n}} \bigg)\ \geq\ \beta,
\end{equation}
where $\beta\in(0,1)$ and $c>0$ are absolute constants.

(ii) Let the conditions of Theorem \ref{th:lower_bounded} be
satisfied. Then
\begin{equation}\label{eq:th:lower_bounded2a}
\inf_{\hat{A}}\ \sup_{{\rm rank}(A_0)\le r}\ \sup_{P_{XY}\in\,{\cal
P}_{A_0,\eta}} \P \bigg(\|\hat{A}-A_0\|_\infty> c\eta \sqrt{m_1m_2}
\sqrt{\frac{m_1 \vee m_2}{n}} \bigg)\ \geq\ \beta,
\end{equation}
where $\beta\in(0,1)$ and $c>0$ are absolute constants.
\end{theorem}

\begin{proof}
Note first that, in the USR matrix completion problem, Assumption
\ref{RI} is satisfied with $\delta_r=0$ and $\mu=\sqrt{m_1m_2}$.

We prove part (i) of the theorem. Consider the set of matrices
$\mathcal A_0$ introduced in the proof of Theorem \ref{th:lower}.
For any two distinct matrices $A_1,A_2$ of $\mathcal A_0$, we have
\begin{equation}\label{eq:lower_spectral_1}
  \|A_1 - A_2\|_\infty\geq \sqrt{\frac{\gamma}{16}}(\sigma\wedge
  a)\sqrt{m_1m_2}\sqrt{\frac{m_1 \vee m_2}{n}}.
\end{equation}
Indeed, if (\ref{eq:lower_spectral_1}) does not hold, we get
$$
\|A_1 - A_2\|_2^2 \leq \mathrm{rank}(A_1-A_2) \|A_1 - A_2\|_\infty^2
< \frac{\gamma}{16}(\sigma\wedge
  a)^2 m_1m_2\frac{(m_1 \vee m_2)r}{n},
$$
since $\mathrm{rank}(A_1-A_2)\leq r$ by construction of $\mathcal
A_0$. This contradicts (\ref{eq: condition A}).

Next, (\ref{eq: condition C}) is satisfied for any $\alpha>0$ if
$\gamma
>0$ is chosen as a sufficiently small numerical constant depending
on $\alpha$.

Combining (\ref{eq:lower_spectral_1}) with (\ref{eq: condition C})
and Theorem 2.5 in \cite{tsy_09} gives the result.

The proof of (ii) follows the same arguments.
\end{proof}

\subsection{Sharp oracle inequalities for the Lasso}\label{subsec:lasso}
As we already mentioned in Example 4 and in the remark after Theorem
\ref{th:main_RE}, one can exploit (\ref{eq:th:main_RE}) to derive
sparsity oracle inequalities for the usual Lasso. This is detailed
in the present subsection. It is noteworthy that the obtained
inequalities are sharp (i.e., with leading constant $1$), which was
not achieved in the previous work on the Lasso.

Note that, if $m_1=m_2=p$ and $A$ and $X_i$ are diagonal matrices,
then the trace regression model (\ref{tracereg1}) becomes
$$
Y_i=x_i^{\top}\beta^*+ \xi_i,\quad i=1,\ldots,n,
$$
where $x_i,\beta^* \in \R^{p}$ denote the vectors of diagonal
elements of $X_i, A_0,$ respectively. Set $\mathbb
X=(x_1,\ldots,x_n)^{\top}\in \mathbb R^{n\times p}$ to be the design
matrix of this linear regression model. For a vector
$z=(z^{(1)},\ldots,z^{(d)})\in \R^d,$  define $|z|_{q}=
\left(\sum_{j=1}^d |z^{(j)}|^q \right)^{1/q}$ for $1\leq q <\infty$
and $|z|_\infty = \max_{1\leq j \leq d}|z^{(j)}|$.

Assume in what follows that $x_i$ are fixed. Then for $A={\rm
diag}(\beta)$ we have $\|A\|_{L_2(\Pi)}^2 = n^{-1}|\mathbb
X\beta|_2^2$ , where $\mathrm{diag}(\beta)$ denotes the diagonal
$p\times p$ matrix with the components of $\beta$ on the diagonal.
We will assume without loss of generality that the diagonal elements
of the Gram matrix $\frac{1}{n}\mathbb X^{\top}\mathbb X$ are not
larger than $1$ (the general case is obtained from this by simple
rescaling).

The estimator $\hat{A}^{\lambda}$ defined in (\ref{ER_fixed})
becomes the usual Lasso estimator
$$
\hat{\beta}^{\lambda} =
\arg\min_{\beta\in\R^p}\left\{\frac{1}{n}\sum_{i=1}^n(Y_i - x_i^\top
\beta)^2 + \lambda |\beta|_1\right\}.
$$

For a vector $\beta\in \R^p,$ we set, with a little abuse of
notation, $\mu_{c_0}(\beta) =\mu_{c_0}(\mathrm{diag}(\beta)),$
$\mu(\beta)=\mu_5(\beta)$. Let $M(\beta)$ denote the number of
nonzero components of $\beta$.

For simplicity, the result is stated only in the case of Gaussian
noise.

\begin{theorem}\label{th:sharpLasso}
Let $\xi_i$ be i.i.d. Gaussian ${\cal N}(0,\sigma^2)$ and let the
diagonal elements of matrix $\frac{1}{n}\mathbb X^{\top}\mathbb X$
be not larger than $1$. Take
$$
\lambda = C\sigma \sqrt{\frac{\log p}{n}},
$$
where $C=3b\sqrt{2}, b\geq 1.$ Then, with probability at least
$1-\frac{1}{p^{b^2-1}\sqrt{\pi\log p}}$, we have
\begin{equation}\label{eq:th:sharpLasso}
\frac{1}{n}|\mathbb X(\hat\beta^{\lambda} - \beta^*)|_2^2 \leq
\inf_{\beta \in \R^p}\left\{ \frac{1}{n}|\mathbb X(\beta -
\beta^*)|_2^2 + C^2\sigma^2\frac{\mu^2(\beta)M(\beta)\log p}{n}
\right\}.
\end{equation}
\end{theorem}

\begin{proof} Combine Theorem \ref{th:main_RE} and a standard bound on the tail of the Gaussian
distribution, which assures that with probability at least
$1-\frac{1}{p^{b^2-1}\sqrt{\pi\log p}}$\,,
$$
\|\mathbf
M\|_\infty=\left|\frac{1}{n}\sum_{i=1}^n\xi_ix_i\right|_\infty \leq
b\sigma \sqrt{\frac{2\log p}{n}}\,.
$$

\end{proof}

Given $\beta\in \R^p$ and $J\subset \{1,\ldots,p\},$ denote by
$\beta_J$ the vector in $\R^p$ which has the same coordinates as
$\beta$ on $J$ and zero coordinates on the complement $J^c$ of $J.$

We recall the Restricted Eigenvalue condition of \cite{BRT09}:

\medskip
\textbf{Condition $\mathbf{RE}(s,c_0)$.} {\it For some integer $s$
such that $1\le s\le p$,  and a positive number $ c_0$ the following
condition holds:}
$$
\kappa(s, c_0) \triangleq \min_{\substack{J\subseteq \{1,\dots,p\},\\
|J|\le s}} \ \ \min_{\substack{u\in \R^p,\, u \neq 0, \\
|u_{J^c}|_1\le c_0|u_{J}|_1}}
\ \ \frac{|\mathbb X u|_2}{\sqrt{n}|u_{J}|_2} \,> \, 0.
$$
\medskip

We have the following corollary.

\begin{corollary} \label{cor:sharpLasso} Let the assumptions of Theorem
\ref{th:sharpLasso} hold, and let condition $\mathbf{RE}(s,5)$ be
satisfied for some $1\leq s \leq p$. Then, with probability at least
$1-\frac{1}{p^{b^2-1}\sqrt{\pi\log p}},$
\begin{equation}\label{eq:cor:sharpLasso}
\frac{1}{n}|\mathbb X(\hat\beta^{\lambda} - \beta^*)|_2^2 \leq
\inf_{\beta \in \R^p\,:\, M(\beta)\leq s}\left\{ \frac{1}{n}|\mathbb
X(\beta - \beta^*)|_2^2 +
\frac{C^2\sigma^2}{\kappa^2(s,5)}\frac{M(\beta)\log p}{n} \right\}.
\end{equation}
\end{corollary}

\begin{proof}
Recall that $e_j(p)$ denote the canonical basis vectors of $\R^p$.
For any $p\times p$ diagonal matrix $A=\mathrm{diag}(\beta)$ with
support $(S_1,S_2),$ $S_1=S_2=\{e_j(p),j\in J\},$ where $J\subset
\{1,\ldots,p\}$ has cardinality $|J|\leq s,$ and an arbitrary
$p\times p$ diagonal matrix $B=\mathrm{diag}(u),$ where $u\in \R^p,$
we have
$$
\|\mathcal{P}_A(B)\|_{1} = |u_J|_1,\quad \|\mathcal{P}_A(B)\|_{2} =
|u_J|_2, \quad \|\mathcal{P}_A^\perp(B)\|_1 = |u_{J^c}|_1
$$
and
$$
\mathbb C_{A,c_0} = \left\{\, \mathrm{diag}(u):\, u\in \R^p,
|u_{J^c}|_1 \leq c_0 |u_J|_1 \right\}, \quad \|B\|_{L_2(\Pi)} =
\frac{1}{\sqrt{n}}|\mathbb X u|_2.
$$
Thus,
\begin{eqnarray*}
\frac1{\mu_{c_0}(A)}&=& \sup_{B\ne 0: \ B\in \mathbb
C_{A,c_0}}\frac{\|B\|_{L_2(\Pi)}}{\|\mathcal{P}_A(B)\|_{2}}
\\ &=& \ \ \min_{\substack{u\in \R^p,\, u \neq 0, \\
|u_{J^c}|_1\le c_0|u_{J}|_1}} \ \ \frac{|\mathbb X
u|_2}{\sqrt{n}|u_{J}|_2} \, %\inf_{u\in\R^p\,:\,
%u\neq 0,\, M(u)\leq s}\ \ \,.
\ge \kappa(s,c_0)
\end{eqnarray*} Since
Condition $\mathbf{RE}(s,5)$ is satisfied, Theorem
\ref{th:sharpLasso} yields the result.
\end{proof}

{\sc Remark~2.} %In the same spirit as in Section
%\ref{subsec:learning},
Oracle inequalities (\ref{eq:th:sharpLasso}) and
(\ref{eq:cor:sharpLasso}) extend straightforwardly to the model
\begin{equation}\label{nonpar}
Y_i=f_i+ \xi_i, \ i=1,\dots,n,
\end{equation}
where $f_i$ are arbitrary fixed values and not necessarily
$f_i=x_i^{\top}\beta^*$. This setting is interesting in the context
of aggregation. Then $x_1,\dots,x_n$ are vectors of values of some
given dictionary of $p$ functions at $n$ given points and $f_i$ are
the values of an unknown regression function at the same points.
Under the model (\ref{nonpar}), inequalities
(\ref{eq:th:sharpLasso}) and (\ref{eq:cor:sharpLasso}) hold true
with the only difference that $\mathbb X\beta^*$ should be replaced
by the vector $f=(f_1,\dots, f_n)^\top $. With such a modification,
(\ref{eq:cor:sharpLasso}) improves upon Theorem 6.1 of \cite{BRT09}
where the leading constant is greater than~1.

\section{Control of the stochastic error}\label{sto}

In this section, we obtain the probability inequalities for the
stochastic error $\|{\bf M}\|_\infty$. For brevity, we will write
throughout $\|\cdot\|_\infty=\|\cdot\|$. The following proposition
is an immediate consequence of the matrix version of Bernstein's
inequality (Corollary 9.1 in \cite{tropp10}).

\begin{proposition}\label{prop:Bernstein_bounded}
Let $Z_1,\ldots,Z_n$ be independent random matrices with dimensions
$m_1\times m_2$ that satisfy $\E(Z_i) = 0$ and $\|Z_i\|\leq U$
almost surely for some constant $U$ and all $i=1,\dots,n$. Define
$$
\sigma_Z = \max\Bigg\{\,\Big\|\frac1{n}\sum_{i=1}^n\E
(Z_iZ_i^\top)\Big\|^{1/2},\, \,\Big\|\frac1{n}\sum_{i=1}^n\E
(Z_i^\top Z_i)\Big\|^{1/2}\Bigg\} .
$$
Then, for all $t>0,$ with probability at least $1-e^{-t}$ we have
$$
\left\| \frac{Z_1 + \cdots + Z_n}{n}  \right\| \leq 2\max\left\{
\sigma_Z\sqrt{\frac{t + \log (m)}{n}}\,, \ \ U \frac{t + \log
(m)}{n} \right\}\,,
$$
where $m=m_1+m_2$.
\end{proposition}
Furthermore, it is possible to replace the $L_\infty$-bound $U$ on
$\|Z\|$ in the above inequality by bounds on the weaker
$\psi_\alpha$-norms of $\|Z\|$ defined by
$$
U_Z^{(\alpha)} = \inf \Big\{u>0:\,  \E\exp(\|Z\|^\alpha/u^\alpha)
\le 2\Big\}, \quad \alpha\geq 1.
$$

\begin{proposition}\label{prop:Bernstein_unbounded}
Let $Z,Z_1,\ldots,Z_n$ be i.i.d. random matrices with dimensions
$m_1\times m_2$ that satisfy $\E(Z) = 0$. Suppose that
$U_Z^{(\alpha)} <\infty$ for some $\alpha\geq 1$. Then there exists
a constant $C>0$ such that, for all $t>0$, with probability at least
$1-e^{-t}$
$$
\left\| \frac{Z_1 + \cdots + Z_n}{n}  \right\| \leq C\max\left\{
\sigma_Z\sqrt{\frac{t + \log (m)}{n}}\,, \ \ U_Z^{(\alpha)}\left(
\log \frac{U_Z^{(\alpha)}}{\sigma_Z} \right)^{1/\alpha} \frac{t +
\log (m)}{n} \right\},
$$
where $m=m_1+m_2$.
\end{proposition}
This is an easy consequence of Proposition 2 in \cite{Kol10}, which
provides an analogous result for Hermitian matrices $Z$. Its
extension to rectangular matrices stated in
Proposition~\ref{prop:Bernstein_unbounded} is straightforward via
the self-adjoint dilation, cf., for example, the proof of
Corollary~9.1 in \cite{tropp10}.

The next lemma gives a control of the stochastic error for USR
matrix completion in the statistical learning setting.

\begin{lemma}\label{lem:completion_bounded}
Let $X_i$ be i.i.d. uniformly distributed on $\mathcal X$. Assume
that $\max_{i=1,\dots,n}|Y_i|\le \eta$ almost surely for some
constant $\eta$. Then for any $t>0$ with probability at least
$1-e^{-t}$ we have
\begin{equation}\label{eq:lem:completion_bounded}
\|{\bf M}\| \le 2\eta\,\max\left\{ \sqrt{\frac{t+\log(m)}{(m_1\wedge
m_2)n}}\,, \ \frac{2(t+\log(m))}{n} \right\}.
\end{equation}
\end{lemma}

\begin{proof} We apply Proposition \ref{prop:Bernstein_bounded} with
$Z_i=Y_iX_i-\E(Y_iX_i)$. Recall that here $X_i$ are i.i.d. with the
same distribution as $X$ and $Y_i$ are not necessarily i.i.d.
Observe that
\begin{eqnarray}\label{eq:norms_of_x}
&&\|X\| = 1,\quad\|\E (X)\| = \sqrt{\frac{1}{m_1m_2}}\ , \quad
\sigma_{X}^2 = \frac1{m_1\wedge m_2} \ .
\end{eqnarray}
Therefore, $\|Z_i\|\le 2\eta$, $\sigma_Z\le \eta\sigma_X $, and the
result follows from Proposition \ref{prop:Bernstein_bounded}.
\end{proof}

We now consider the USR matrix completion with sub-exponential
errors. Recall that in this case we assume that the pairs
$(X_i,Y_i)$ are i.i.d. We have
\begin{eqnarray*}
\|{\bf M}\| &=& \left\| \frac{1}{n}\sum_{i=1}^n (Y_i X_i - \mathbb E
(Y_iX_i)) \right\|\\ &\leq& \left\| \frac{1}{n}\sum_{i=1}^n \xi_i
X_i \right\| + \left\| \frac{1}{n}\sum_{i=1}^n
\Big(\mathrm{tr}(A_0^\top X_i) X_i -
\mathbb E (\mathrm{tr}(A_0^\top X)X)\Big)  \right\|\\
&=& \Delta_1 + \Delta_2.
\end{eqnarray*}
We treat the terms $\Delta_1$ and $\Delta_2$ separately in the two
lemmas below.

\begin{lemma}\label{lem:tau1}
Let $X_i$ be i.i.d. uniformly distributed on $\mathcal X$, and the
pairs $(X_i,Y_i)$ be i.i.d. Assume that condition (\ref{subexp})
holds. Then there exists an absolute constant $C>0$ that can depend
only on $\alpha,c_1,\tilde c$ and such that, for all $t>0,$ with
probability at least $1-2e^{-t}$ we have
\begin{equation}\label{eq:lem:tau1}
\Delta_1 \le C\sigma\max\left\{ \sqrt{\frac{t+\log(m)}{(m_1\wedge
m_2)n}}\, , \ \frac{(t+\log(m))\log^{1/\alpha}(m_1\wedge m_2)}{n}
\right\}\,.
\end{equation}
\end{lemma}

\begin{proof} Observe first that for ${\tilde
X}=X-\E(X)$ we have
\begin{eqnarray}\label{eq:norms_of_x_1}
&& \sigma_{\tilde X}^2= \frac1{m_1\wedge m_2}\,.
\end{eqnarray}
Now,
  \begin{eqnarray}\label{interm1}
    \left\| \frac{1}{n}\sum_{i=1}^n \xi_i X_i\right\|
    &\leq& \left\| \frac{1}{n}\sum_{i=1}^n \xi_i (X_i-\E X_i)\right\| + \left\| \frac{1}{n}\sum_{i=1}^n \xi_i \E(X_i) \right\|\nonumber\\
    %&\leq& \left\| \frac{1}{n}\sum_{i=1}^n \xi_i (X_i-\E X)\right\| +\|\E(X)\| \left| \frac{1}{n}\sum_{i=1}^n \xi_i
    %\right|\nonumber\\
    &\leq& \left\| \frac{1}{n}\sum_{i=1}^n \xi_i \left(X_i-\E X\right)\right\| + \sqrt{\frac{1}{m_1m_2}}\left| \frac{1}{n}\sum_{i=1}^n \xi_i
    \right|.
  \end{eqnarray}
Set $Z_i=\xi_i \left(X_i-\E X\right)$. These are i.i.d. random
matrices having the same distribution as a random matrix $Z$. It
follows from (\ref{eq:norms_of_x}) that $\|Z_i\|\le 2|\xi_i|$, and
thus condition (\ref{subexp}) implies that $U_Z^{(\alpha)}\leq c
\sigma$ for some constant $c>0$. Furthermore, in view of
(\ref{eq:norms_of_x_1}), we have $\sigma_Z\le c'\sigma
\sigma_{\tilde X}= c' \sigma/(m_1\wedge m_2)^{1/2}$ for some
constant $c'>0$ and $\sigma_Z\ge c_1^{1/2}\sigma/(2(m_1\wedge
m_2))^{1/2}$. Using these remarks we can deduce from Proposition
\ref{prop:Bernstein_unbounded} that there exists an absolute
constant $\tilde{C}>0$ such that for any $t>0$ with probability at
least $1-e^{-t}$ we have
\begin{align*}
   &\left\| \frac{1}{n}\sum_{i=1}^n \xi_i
\left(X_i-\E X\right)\right\|\\
&\hspace{0.5cm}\leq \tilde{C}\max\left\{\sigma_{Z}
\sqrt{\frac{t+\log (m)}{n}}\,, \ \ U_Z^{(\alpha)}\left(\log
    \frac{U_Z^{(\alpha)}}{\sigma_Z}\right)^{1/\alpha}\frac{t +
    \log (m)}{n}\right\}
    \\
    &\hspace{0.5cm}\leq C\sigma\max\left\{ \sqrt{\frac{t+\log(m)}{(m_1\wedge
m_2)n}}\, , \ \frac{(t+\log(m))\log^{1/\alpha}(m_1\wedge m_2)}{n}
\right\}\,.
\end{align*}
Finally, in view of Condition (\ref{subexp}) and Bernstein's
inequality for sub-exponential noise, we have for any $t>0$, with
probability at least $1-e^{-t}$,
\begin{eqnarray*}
\left| \frac{1}{n}\sum_{i=1}^n\xi_i\right| &\leq& C\sigma
\max\left\{ \sqrt{\frac{t}{n}} , \frac{t}{n}\right\},
\end{eqnarray*}
where $C>0$ depends only on $\tilde c$. We complete the proof by
using the union bound.
\end{proof}

Define now
\begin{align*}
|A_0|_* &= \max \left\{ \,\sqrt{\max_{1\leq i \leq
m_1}\sum_{j=1}^{m_2} a_0^2(i,j)},\ \sqrt{\max_{1\leq j \leq
m_2}\sum_{i=1}^{m_1} a_0^2(i,j)} \ \right\}\\
&\hspace{7cm}\le \max_{i,j}|a_0(i,j)| \sqrt{m_1\vee m_2}\,.
\end{align*}

\begin{lemma}\label{lem:tau2}
Let $X_i$ be i.i.d. random variables uniformly distributed in
$\mathcal X$. Then, for all $t>0,$ with probability at least $1
  -e^{-t}$ we have
\begin{align}\label{eq:lem:tau2}
\Delta_2  \le 2\max \left\{
\,|A_0|_*\sqrt{\frac{t+\log(m)}{m_1m_2n}}\,, \ \,
2\max_{i,j}|a_0(i,j)|\frac{t+\log(m)}{n}\ \right\}\,.
\end{align}
If $\max_{i,j}|a_0(i,j)|\leq a$ for some $a>0$, then with the same
probability
$$
\Delta_2  \le  \, 2a \, \max \left\{
\,\sqrt{\frac{t+\log(m)}{(m_1\wedge m_2)n}}, \
\frac{2(t+\log(m))}{n} \ \right\}\,.
$$

\end{lemma}
\begin{proof} We apply Proposition \ref{prop:Bernstein_bounded} for
the random variables $Z_i=\mathrm{tr}(A_0^\top X_i) X_i - \mathbb E
(\mathrm{tr}(A_0^\top X)X)$. Using (\ref{eq:norms_of_x}) we get
$\|Z_i\|\le 2\max_{i,j}|a_0(i,j)|$ and
$$
\sigma^2_Z \le \max \left\{ \, \|\E\big(\langle A_0,X  \rangle^2
XX^\top\big)\|,\, \|\E\big(\langle A_0,X  \rangle^2 X^\top
X\,\big)\| \,\right\}\le \frac{|A_0|_*^2}{m_1m_2}\,.
$$
Thus, (\ref{eq:lem:tau2}) follows from Proposition
\ref{prop:Bernstein_bounded}.

\end{proof}


\begin{thebibliography}{99}
\footnotesize{


%\bibitem{Audibert} Audibert, J.-Y. (2004) Une approche PAC-bay\'esienne de
%la th\'eorie statistique de l'apprentissage. PhD Thesis, University of
%Paris 6.






%\bibitem{BRT} Bickel, P., Ritov, Y. and Tsybakov, A. (2007)
%Simultaneous Analysis of LASSO and Dantzig Selector. Preprint.
%

%\bibitem{BTW1}
%Bunea, F., Tsybakov, A. and Wegkamp, M.
%(2007a) Sparsity oracle inequalities for the LASSO.
%\it Electronic Journal of Statistics, \rm 1, 169--194.

%\bibitem{BTW2}
%Bunea, F., Tsybakov, A. and Wegkamp, M. (2007b)
%Sparse density estimation with $\ell_1$ penalties.
%In: \it Proc. 20th Annual Conference on Learning Theory (COLT 2007),
%\rm
%Lecture Notes in Artificial Intelligence, Springer, v. 4539,
%pp. 530--543.


%\bibitem{CT-2} Candes, E. and Tao, T. (2007) The Dantzig selector:
%statistical estimation when $p$ is much larger than $n.$
%\it Annals of Statistics, \rm 35, 6, 2313-2351.



%\bibitem{Cat} Catoni, O. (2004) Statistical Learning Theory
%and Stochastic Optimization. \it Ecole d'Et\'e de Probabilit\'es de
%Saint-Flour XXXI -2001, \rm Lecture Notes in Mathematics, {\bf
%1851}, Springer, New York.

%\bibitem{DalTsy} Dalalyan, A. and Tsybakov, A. (2007)
%Aggregation
%by exponential weighting and sharp oracle inequalities.
%In: \it Proc. 20th Annual Conference on Learning Theory (COLT 2007),
%\rm
%Lecture Notes in Artificial Intelligence, Springer, v. 4539,
%pp. 97--111.


%\bibitem{Don2} Donoho, D.L. (2006) For Most Large Underdetermined Systems
%of Linear Equations the Minimal $\ell^1$-norm Solution is also
%the Sparsest Solution. \it Communications on Pure and Applied
%Mathematics, \rm 59, 797--829.


%\bibitem{van de Geer} van de Geer, S. (2007) High-dimensional generalized
%linear models and the Lasso, \it Annals of Statistics, \rm to appear.



%\bibitem{KP} Koltchinskii, V. and Panchenko, D. (2005) Complexities
%of convex combinations and bounding the generalization error in
%classification. \it Annals of Statistics, \rm 33, no 4.

%\bibitem{Kol1} Koltchinskii, V. (2005) Model selection and aggregation
%in sparse classification problems. {\it Oberwolfach Reports: Meeting
%on Statistical and Probabilistic Methods of Model Selection,
%October 2005}.

%\bibitem{Kol} Koltchinskii, V. (2006) Local Rademacher Complexities
%and Oracle Inequalities in Risk Mnimization, \it Annals of
%Statistics, \rm 34, 6, 2593--2656.

%\bibitem{Adamczak} Adamczak, R. (2008) A tail inequality for suprema
%of unbounded empirical processes with applications to Markov chains.
%\it Electronic Journal of Probability, \rm 13, 34, 1000--1034.
%
%\bibitem{Gill} Artiles, L.M., Gill, R. and Guta, M.I.(2004) An invitation
%to quantum tomography. \it J. Royal Statistical Society,\rm Ser. B,
%v. 67, 1, 109--134.

\bibitem{Ahlswede} \textsc{Ahlswede, R. and Winter, A.} (2002) Strong converse for
identification via quantum channels. \it IEEE Transactions on
Information Theory, \rm 48, 3, pp. 569--679.

\bibitem{argyr08} \textsc{Argyriou, A., Evgeniou, T. and Pontil, M.}
 (2008) \newblock Convex multi-task feature
learning. \newblock \textsl{Machine Learning, \textbf{73}},
243--272.

\bibitem{argyr09} \textsc{Argyriou, A., Micchelli, C.A., and Pontil, M.}
 (2010) \newblock On spectral
learning. \newblock \textsl{Journal of Machine Learning Research},
11, 935-953.

\bibitem{argyr07} \textsc{Argyriou, A., Micchelli, C.A., Pontil, M.,
and Ying, Y.}
 (2007) \newblock A Spectral Regularization Framework for
Multi-Task Structure Learning. \newblock \textsl{Proceedings of
NIPS-2007}.

\bibitem{Aubin} \textsc{Aubin, J.P., and Ekeland, I.} (1984)
\newblock \textsl{Applied Nonlinear Analysis.}
    \newblock Wiley, New York.

\bibitem{bach08} \textsc{Bach, F.R.} (2008). \newblock Consistency
of trace norm minimization. \newblock \textsl{Journal of Machine
Learning Research, \textbf{9}}, 1019--1048.

%\bibitem{Bhatia} Bhatia, R. (1997) Matrix Analysis. Springer, New York.

\bibitem{BRT09}
\textsc{Bickel, P., Ritov, Y. and Tsybakov, A.~B.} (2009).
Simultaneous analysis of {L}asso and {D}antzig selector.
\textit{Ann.\ Statist.} \textbf{37} 1705--1732.


\bibitem{bsw10}
\textsc{Bunea, F., She, Y. and Wegkamp, M.H.} (2010)
\newblock Optimal selection of reduced rank estimators of high-dimensional matrices. \newblock
 \texttt{arXiv:1004.2995}. April 2010.

\bibitem{cp09}
\textsc{Cand\`es, E. J. and  Plan, Y.} (2009) \newblock Matrix
completion with noise. {\it Proceedings of IEEE}, 2009.


\bibitem{Candes_Plan}
\textsc{Cand\`es, E. J. and  Plan, Y.} (2010) \newblock Tight oracle
bounds for low-rank matrix recovery from a mininal number of noisy
random measurements. \texttt{arXiv:1001.0339}. January, 2010.


\bibitem{Candes_Recht} \textsc{Cand\`es, E. J. and Recht, B.} (2009) Exact matrix
completion via convex optimization. \it Foundations of Computational
Mathematics, \rm 9(6), 717--772.

\bibitem{Candes_Tao} \textsc{Cand\`es, E. and Tao, T.} (2009) The power of
convex relaxation: Near-optimal matrix completion.
\texttt{arXiv:0903.1476}

\bibitem{gir10} \textsc{Giraud, C.} (2010) Low rank multivariate regression.
\texttt{arXiv:1009.5165}

\bibitem{Gaiffas_Lecue} \textsc{Gaiffas, S. and Lecu\'e, G.} (2010)
Sharp oracle inequalities for the prediction of a high-dimensional
matrix. \texttt{arXiv:1008.4886}



%\bibitem{Gross-1} Gross, D., Lou, Yo-Kai, Flammia, S.T.,
%Becker, S. and Assert, J. (2009) Quantum State Tomography via
%compressed sensing. Preprint.

\bibitem{Gross-2} \textsc{Gross, D.} (2009) Recovering low-rank
matrices from few coefficients in any basis.
\texttt{arXiv:0910.1879}.


%\bibitem{Klartag-Mendelson} Klartag, B. and Mendelson, S. (2005) Empirical
%Processes and Random Projections. \it Journal of Functional
%Analysis, \rm 225(1), 229--245.

%\bibitem{Klauck} Klauck, H., Nayak, A., Ta-Shma, A. and Zuckerman, D.
%(2007) Interactions in Quantum Communication. \it IEEE Transactions
%on Information Theory, \rm 53, 6, 1970--1982.

\bibitem{Keshavan} \textsc{Keshavan, R.H., Montanari, A. and Oh, S.}
(2009) Matrix completion from noisy entries.
\texttt{arXiv:0906.2027}

%\bibitem{Kol-09} Koltchinskii, V. (2009) Sparse recovery in convex hulls
%via entropy penalization. \it Annals of Statistics, \rm 37(3),
%1332--1359.

\bibitem{kol_dantz} \textsc{Koltchinskii, V.} (2009) The Dantzig Selector
and Sparsity Oracle Inequalities. \it Bernoulli, \rm {\bf 15},
799-828.

\bibitem{Kol10} \textsc{Koltchinskii, V.} (2010) von Neumann entropy
penalization and low rank matrix approximation.
\texttt{arXiv:1009.2439}

%\bibitem{LT} Ledoux, M. and Talagrand, M. (1991) Probability in Banach
%Spaces. Springer.


%\bibitem{Mendelson} Mendelson, S. (2010) Empirical processes with a
%bounded $\psi_1$ diameter. Preprint.

\bibitem{nw09}
\textsc{Negahban, S. and Wainwright, M.J.} (2009)
\newblock Estimation of (near) low rank matrices with noise and high-dimensional scaling.
\newblock  \texttt{arXiv:0912.5100}, December 2009.

\bibitem{nw10}
\textsc{Negahban, S. and Wainwright, M.J.} (2010) Restricted strong
convexity and weighted matrix completion: Optimal bounds with noise.
\texttt{arXiv:1009.2118}

\bibitem{negaban10} \textsc{Negahban, S., Ravikumar, P.,
Wainwright, M.J., and Yu, B.} (2010)
 A unified framework for high-dimensional analysis of
 $M$-estimators with decomposable regularizers.
 \texttt{arXiv:1010.2731}


%\bibitem{Nielsen} Nielsen, M.A. and Chang, I.L. (2000) Quantum Computation and Quantum Information, Cambridge University Press.

\bibitem{Recht}
\textsc{Recht, B.} (2009) A simpler approach to matrix completion.
\texttt{arXiv:0910.0651}

\bibitem{rfp07}
\textsc{Recht, B., Fazel, M. and Parrilo, P.A.} (2007)
\newblock Guaranteed Minimum-Rank Solutions of Linear Matrix
Equations via Nuclear Norm Minimization. \texttt{arXiv:0706.4138}

%\bibitem{R-V} Rudelson, M. and Vershynin, R. (2010)
%Non-asymptotic theory of random matrices: extreme singular values.
%\it Proceedings of the International Congress of Mathematicians, \rm
%Hyderabad, India.

\bibitem{Rohde} \textsc{Rohde, A. and Tsybakov, A.} (2009) Estimation of
high-dimensional low rank matrices. \texttt{arXiv:0912.5338}
December 2009.

\bibitem{SS90}
\textsc{Stewart, G. W., Sun, J.} (1990), Matrix Perturbation Theory.
New York, Academic Press

\bibitem{tropp10}
\textsc{Tropp, J. A.} (2010) \newblock User-friendly tail bounds for
sums of random matrices.
\newblock \texttt{arXiv:1004.4389}, April 2010.

\bibitem{tsy_09}
    \textsc{Tsybakov, A.} (2009)
    \newblock \textsl{Introduction to Nonparametric Estimation.}
    \newblock Springer.

\bibitem{watson}
    \textsc{Watson, G. A.} (1992)
    \newblock Characterization of the subdifferential of
    some matrix norms.
    \newblock \textsl{Linear Algebra Appl.}, 170, 33-45.
%\bibitem{Yu} Yu, Zhu, Lafferty and Gong (2009).

%\bibitem{Simon} Simon, B. (1979) Trace Ideals and their Applications.
%Cambridge University Press.

%\bibitem{Tal} Talagrand, M. (2005) The Generic Chaining. Springer.

%\bibitem{Wellner} van der Vaart, A. and Wellner, J. (1996) Weak Convergence and Empirical Processes. With Applications to Statistics.
%Springer.

%\bibitem{Kol2_a} Koltchinskii, V. (2007a) The Dantzig Selector and
%Sparsity Oracle Inequalities. Preprint.




%\bibitem{Massart-06} Massart, P. (2007) Concentration
%Inequalities and Model Selection. \it
%Ecole d'ete de Probabilit\'es de Saint-Flour 2003, \rm
%Lecture Notes in Mathematics, Springer.


%\bibitem{McAllester} McAllester, D.A. (1998) Some PAC-Bayesian
%theorems. In \it Proc. 11th Annual Conference on Learning Theory, \rm
%pp. 230--234, ACM Press.


%\bibitem{MPT-J} Mendelson, S., Pajor, A. and Tomczak-Jaegermann, N.
%(2005) Reconstruction and subgaussian operators in Asymptotic
%Geometric Analysis, \it Geometric and Functional Analysis, \rm
%to appear.




%\bibitem{RV} Rudelson, M. and Vershynin, R. (2005) Geometric Approach
%to Error Correcting Codes and Reconstruction of Signals.
%\it Int. Math. Res. Not., \rm 64, 4019--4041.


%\bibitem{Tsy} Tsybakov, A. (2003) Optimal rates of aggregation. In:
%\it Proc. 16th Annual Conference on Learning Theory (COLT) and 7th
%Annual Workshop on Kernel Machines, \rm Lecture Notes in Artificial
%Intelligence, {\bf 2777}, Springer, New York, 303--313.


%\bibitem{Yang1} Yang, Y. (2000) Mixing strategies for density
%estimation. \it Annals of Statistics, \rm {\bf 28}, 75--87.

%\bibitem{Znang} Zhang, T. (2001) Regularized Winnow Method. In:
%\it Advances in Neural Information Processing Systems 13 (NIPS2000),
%\rm T.K. Leen, T.G. Dietrich and V. Tresp (Eds), MIT Press, pp. 703--709.



%\bibitem{Zhang1} Zhang, T. (2006a) From epsilon-entropy to KL-complexity:
%analysis of minimum information complexity density estimation. \it
%Annals of Statistics, \rm 34, 2180--2210.

%\bibitem{Zhang2} Zhang, T. (2006b) Information Theoretical Upper and
%Lower Bounds for Statistical Estimation. \it IEEE Transactions on
%Information Theory, \rm 52, 1307-1321.

}

\end{thebibliography}
\end{document}